\documentclass[a4paper,10pt]{article}
\usepackage{amsmath}  
\usepackage{amssymb}            
\usepackage{amsxtra}
\usepackage{amscd}
\usepackage[mathscr]{euscript}
\usepackage[dvipdfmx]{graphicx}
\usepackage{tikz}
\usepackage{tikz-cd}

\everymath{\displaystyle} 

\newtheorem{thm}{Theorem}[subsection]
\newtheorem{prop}[thm]{Proposition}

\newtheorem{dfn}[thm]{Definition}

\newtheorem{rem}[thm]{Remark}
\newtheorem{conj}[thm]{Conjecture}

\numberwithin{equation}{section}

\newcommand{\npmod}[1]{\!\!\pmod{#1}}
\newcommand{\nnpmod}[1]{\!\!\!\!\pmod{#1}}

\newenvironment{proof}{\par\noindent{\bf[Proof]}}%
                      {$\blacksquare$\noindent\par\vspace{0.5\baselineskip}}
                      {$\blacksquare$\par\noindent}

\font\b=cmr10 scaled \magstep4

\def\bigzerou{\smash{\lower0.7ex\hbox{\b 0}}}
\def\bigastl{\smash{\lower0.7ex\hbox{\b *}}}
\def\bigastu{\smash{\lower2.7ex\hbox{\b *}}}

\def\addots{\mathinner
    {\mkern1mu\raise1pt\hbox{.}\mkern2mu
    \raise4pt\hbox{.}\mkern2mu\raise7pt\vbox{\kern7pt\hbox{.}}\mkern1mu}}

\makeatletter
\newcommand{\numsubsection}{\@startsection%
  {subsection}%
  {2}%
  {0mm}%
  {\baselineskip}%
  {-0.2\parindent}%
{\normalfont\large\upshape\bfseries}}%
\def\@dotsep{1.5}
\def\@pnumwidth{1em}
\makeatother



\title{On certain supercuspidal representations of $SL_n(F)$ 
  associated with\\ tamely ramified extensions :\\
  the formal degree conjecture and\\
  the root number conjecture} 
\author{Koichi Takase
        \thanks{The author is partially supported by 
                    JSPS KAKENHI Grant Number JP 16K05053}}
\date{}

\begin{document}   

%\begin{frontmatter}

%\title{}
%\author{Koichi Takase}
%\address{Sendai 980-0845, Japan\\
%         Miyagi University of Education\\
%         Department of Mathematics}
%
%\begin{abstract}
%\input{abstract.tex}
%\end{abstract}
%
%\end{frontmatter}

\maketitle

%\begin{abstract}
%\input{../abstract.tex}
%\end{abstract}

\section{Introduction}
\label{sec:introduction}

\numsubsection{}
\label{subsec:conjectures-for-split-semi-simple-group}
Let $F/\Bbb Q_p$ be a finite extension with $p\neq 2$ whose 
integer ring $O_F$ has unique maximal ideal $\frak{p}_F$ wich
is generated by $\varpi_F$. The residue class field 
$\Bbb F=O_F/\frak{p}_F$ is a finite field of $q$-elements. The Weil
  group of $F$ is denoted by $W_F$ which is a subgroup of the absolute
  Galois group $\text{\rm Gal}(\overline F/F)$ where 
$\overline F$ is a fixed algebraic closure of $F$ in which we will
  take the algebraic extensions of $F$.

Let $G$ be a connected semi-simple linear algebraic group defined over
$F$. For the sake of simplicity, we will assume that $G$ splits over
$F$. Then the $L$-group $^LG$ of $G$ is equal to the dual group 
$G\sphat$ of $G$. 
An admissible representation
$$
 \varphi:W_F\times SL_2(\Bbb C)\to\,^LG
$$
of the Weil-Deligne group of $F$ is called a discrete parameter of $G$
over $F$ if the 
centralizer $\mathcal{A}_{\varphi}=Z_{\,^LG}(\text{\rm Im}\varphi)$ of
the image of $\varphi$ in  $^LG$ is a finite
group. Let us denote by $\mathcal{D}_F(G)$ the $G\sphat$-conjugacy classes
of the discrete parameters of $G$ over $F$. The conjectural
parametrization of $\text{\rm Irr}_2(G)$ (resp. $\text{\rm Irr}_s(G)$),
the set of the equivalence 
classes of the irreducible admissible square-integrable (resp. 
supercuspidal) representations of $G$, by $\mathcal{D}_F(G)$ is 
(see \cite[p.483, Conj.7.1]{Gross-Reeder2010} for the details)

\begin{conj}
\label{conjecture:lamglamds-parameter-of-square-integrable-rep}
For every $\varphi\in\mathcal{D}_F(G)$, there exists a finite subset 
$\Pi_{\varphi}$ of $\text{\rm Irr}_2(G)$ such that
\begin{enumerate}
\item $\text{\rm Irr}_2(G)
       =\bigsqcup_{\varphi\in\mathcal{D}_F(G)}\Pi_{\varphi}$,
\item there exists a bijection of $\Pi_{\varphi}$ onto the equivalence
  classes $\mathcal{A}_{\varphi}\sphat$ 
  of the irreducible complex linear representations of 
  $\mathcal{A}_{\varphi}$,
\item $\Pi_{\varphi}\subset\text{\rm Irr}_s(G)$ if 
      $\varphi|_{SL_2(\Bbb C)}=1$.
\end{enumerate}
The finite set $\Pi_{\varphi}$ is called a $L$-packet of $\varphi$.
\end{conj}

According to this conjecture, any $\pi\in\text{\rm Irr}_2(G)$ is
determined by $\varphi\in\mathcal{D}_F(G)$ and 
$\sigma\in\mathcal{A}_{\varphi}\sphat$. So the formal degree of $\pi$
should be determined by these data. The formal degree conjecture due
to Hiraga-Ichino-Ikeda \cite{HiragaIchinoIkeda2008} is 
(with the formulation of \cite{Gross-Reeder2010}) 

\begin{conj}\label{conjecture:formal-degree-conjecture}
The formal degree $d_{\pi}$
of $\pi$ with respect to the absolute value of the Euler-Poincar\'e
measure (see {\rm \cite[$\S 3$]{Serre1971}} for the details) 
on $G(F)$ is equal to
$$
 \frac{\dim\sigma}
      {|\mathcal{A}_{\varphi}|}\cdot
 \left|
 \frac{\gamma(\varphi,\text{\rm Ad},\psi,d(x),0)}
      {\gamma(\varphi_0,\text{\rm Ad},\psi,d(x),0)}\right|.
$$
\end{conj}

Here 
$$
 \gamma(\varphi,\text{\rm Ad},\psi,d(x),s)
 =\varepsilon(\varphi,\text{\rm Ad},d(x),s)\cdot
  \frac{L(\varphi^{\vee},\text{\rm Ad},1-s)}
       {L(\varphi,\text{\rm Ad},s)}
$$
is the gamma-factor associated with the $\varphi$ combined with the
adjoint representation $\text{\rm Ad}$ of $G\sphat$ on 
its Lie algebra $\frak{g}\sphat$, and a continuous
additive character $\psi$ of $F$ such that 
$\{x\in F\mid\psi(xO_F)=1\}=O_F$ 
and the Haar measure $d(x)$ on the additive group $F$ such that 
$\int_{O_F}d(x)=1$. See \cite[pp.440-441]{Gross-Reeder2010} for the details.
$$
 \varphi_0:W_F\times SL_2(\Bbb C)
           \xrightarrow{\text{\rm proj.}}SL_2(\Bbb C)
           \to G\sphat
$$
is the principal parameter (see \cite[p.447]{Gross-Reeder2010} for the
definition). 

The formal degree conjecture concerns with the absolute value
of the epsilon-factor
$$
 \varepsilon(\varphi,\text{\rm Ad},d(x),s)
 =w(\varphi,\text{\rm Ad})\cdot q^{a(\varphi,\text{\rm Ad})(\frac 12-s)}
$$
where $a(\varphi,\text{\rm Ad})$ is the Artin-conductor and 
$w(\varphi,\text{\rm Ad})$ is the root number. 

In order to state the root number conjecture, we need some
notations. Let $T\subset G$ be a maximal torus split over $F$ with
respect to which the root datum 
$$
 (X(T),\Phi(T),X^{\vee}(T),\Phi^{\vee}(T))
$$
is defined. Then the dual group $G\sphat$ is, by the definition, the
connected reductive complex algebraic group with a maximal torus
$T\sphat$ with which its root datum is isomorphic to
$$
 (X^{\vee}(T),\Phi^{\vee}(T),X(T),\Phi(T)).
$$
Put $2\cdot\rho=\sum_{0<\alpha\in\Phi^{\vee}(T)}\alpha$, then 
$\epsilon=2\cdot\rho(-1)\in T$ is a central element of $G$. Now the
root number conjecture says that 

\begin{conj}{\rm \cite[p.493, Conj.8.3]{Gross-Reeder2010}}
\label{conjecture:root-number-conjecture}
$$
 \frac{w(\varphi,\text{\rm Ad})}
      {w(\varphi_0,\text{\rm Ad})}=\pi(\epsilon)
$$ 
where $\epsilon$ is the central element of $G$ defined above 
(see {\rm \cite[p.492, (65)]{Gross-Reeder2010}} for the details).
\end{conj}

Since $G$ is assumed to be splits over $F$, we have 
$w(\varphi_0,\text{\rm Ad})=1$ (see \cite[p.448]{Gross-Reeder2010}).

\numsubsection{}
In this paper, we will construct quite explicitly supercuspidal
representations of $G(F)=SL_n(F)$ associated with a tamely ramified
extension $K/F$ of degree $n$ (Theorem 
\ref{th:supercuspidal-representation-of-sl(n)}). When $K/F$ is normal, 
we will also give  candidates of Langlands parameters of 
the supercuspidal representations (the section 
\ref{sec:kaleta-l-parameter}), and will verify the validity of the
formal degree conjecture (Theorem
\ref{th:formal-degree-conjecture-for-sl(n)}) 
and the root number conjecture (Theorem 
\ref{th:root-number-conjecture}) with them.

Our supercuspidal representations, denoted by $\pi_{\beta,\theta}$,
are given by the compact induction 
$\text{\rm ind}_{G(O_F)}^{G(F)}\delta_{\beta,\theta}$ from irreducible
unitary representations $\delta_{\beta,\theta}$ of the hyperspecial
compact subgroup $G(O_F)=SL_n(O_F)$. Here $\pi_{\beta,\theta}$ and
$\delta_{\beta,\theta}$ are characterized each other by the conditions
\begin{enumerate}
\item $\delta_{\beta,\theta}$ factors through the canonical morphism 
      $G(O_F)\to G(O_F/\frak{p}_F^r)$ with $r\geq 2$, and 
      the multiplicity of $\delta_{\beta,\theta}$ in 
      $\pi_{\beta,\theta}|_{G(O_F)}$ is one, 
\item any irreducible unitary representation $\delta$ of $G(O_F)$
      which factors through the canonical morphism 
      $G(O_F)\to G(O_F/\frak{p}_F^r)$, and a constituent of 
      $\pi_{\beta,\theta}|_{G(O_F)}$, then
      $\delta=\delta_{\beta,\theta}$. 
\end{enumerate}
The parameters $\beta$ and $\theta$ are associated with the tamely
ramified extension $K/F$, that is, $O_K=O_F[\beta]$ and 
$\theta$ is a certain continuous unitary character of 
$$
 U_{K/F}=\{x\in K^{\times}\mid N_{K/F}(x)=\}
$$
(see the subsection 
\ref{subsec:symplectic-space-associated-with-tamely-ramified-ext} for
the precise definitions). 
We have the irreducible representation $\delta_{\beta,\theta}$ by the
general theory given by \cite{Takase2021}. 

The candidate of Langlands parameter is given by the method of Kaletha 
\cite{Kaletha2019}. Regard the compact group $U_{K/F}$ as the
group of $F$-rational points of an elliptic torus of $SL_n$. Then, by
the local Langlands correspondence of tori 
(see \cite{Yu2009}) and the Langlands-Schelstad procedure 
(\cite{LanglandsShelstad1987}) gives a group homomorphism $\varphi$ of
the Weil group $W_F$ of $F$ to the dual group 
$G\sphat=PGL_n(\Bbb C)$ of $SL_n$ over $F$. 

\numsubsection{}
The section \ref{sec:supercuspidal-representation-of-sl(n)} is devoted
to the construction of the supercuspidal representation 
$\pi_{\beta,\theta}$ of $SL_n(F)$. After recalling, in the subsection 
\ref{subsec:regular-irred-character-of-hyperspecial-compact-subgroup}, 
the general theory
of the regular irreducible representations of the finite group 
$G(O_F/\frak{p}_F^r)$ ($r\geq 2$) given by 
\cite{Takase2021}, we will define the irreducible unitary
representation $\delta_{\beta,\theta}$ of $SL_n(O_F)$ in the
subsection
\ref{subsec:symplectic-space-associated-with-tamely-ramified-ext}. The
construction of the supercuspidal representation $\pi_{\beta,\theta}$
is given in the subsection
\ref{subsec:construction-of-supercuspidal-representataion}. 

The candidate of Langlands parameter is given in the section 
\ref{sec:kaleta-l-parameter}. The local Langlands correspondence of
elliptic torus (Proposition
\ref{prop:local-langlands-correspondence-of-elliptic-tori}), the
Langlands-Schelstad procedure (the subsection 
\ref{subsec:chi-datum}) are given quite explicitly. In particular, the
candidate of Langlands parameter is given by
$$
 \varphi:W_F\to W_{K/F}
            \xrightarrow{(\ast)} GL_n(\Bbb C)\to PGL_n(\Bbb C)
$$
where 
$(\ast)=\text{\rm Ind}_{K^{\times}}^{W_{K/F}}\widetilde\vartheta$ is
the induced representation from a character $\widetilde\vartheta$ of
$K^{\times}$ to the relative Weil group 
$W_{K/F}=W_F/\overline{[W_K,W_K]}$. 

Using the explicit description of the parameter $\varphi$, we will
verify the formal degree conjecture in the section 
\ref{sec:formal-degree-conjecture}, and the root number conjecture in
the section \ref{sec:root-number-conjecture}. 

Several basic facts on the local factor associated with
representations of the Weil group are given in the appendix 
\ref{sec:local-factors}.

\section{Supercuspidal representations of $SL_n(F)$}
\label{sec:supercuspidal-representation-of-sl(n)}

\subsection{Regular irreducible characters of hyperspecial compact
            subgroup}
\label{subsec:regular-irred-character-of-hyperspecial-compact-subgroup}
Let us recall the main results of \cite{Takase2021}.

Fix a continuous unitary additive character $\psi:F\to\Bbb C^1$ such
that
$$
 \{x\in F\mid\psi(xO_F)=1\}=O_F.
$$
Let $G=SL_n$ be the $O_F$-group scheme such that, for any
$O_F$-algebra 
\footnote{In this paper, an $O_F$-algebra means an unital commutative
  $O_F$-algebra.} 
$R$, the group of the $A$-valued point is $G(A)=SL_n(A)$. 
Let $\frak{g}$ the Lie algebra scheme of $G$ which is a closed affine
$O_F$-subscheme of $\frak{gl}_n$ the Lie algebra scheme of $GL_n$
defined by
$$
 \frak{g}(R)=\{X\in\frak{gl}_n(R)\mid \text{\rm tr}(X)=0\}
$$
for all $O_F$-algebra $R$. 
Let 
$$
 B:\frak{gl}_n{\times}_{O_F}\frak{gl}_n\to\Bbb A_{O_F}^1
$$
be the trace form on $\frak{gl}_n$, that is $B(X,Y)=\text{\rm tr}(XY)$
for all $X,Y\in\frak{gl}_n(R)$ with any $O_F$-algebra $R$. Since $G$
is smooth $O_F$-group scheme, we have a canonical isomorphism
$$
 \frak{g}(O_F)/\varpi^r\frak{g}(O_F)\,\tilde{\to}\,
 \frak{g}(O_F/\frak{p}^r)=\frak{g}(O_F){\otimes}_{O_F}O_F/\frak{p}^r
$$
(\cite[Chap.II, $\S 4$, Prop.4.8]{Demazure-Gabriel1970}) and the
 canonical group homomorphism $G(O_F)\to G(O_F/\frak{p}^r)$ is
 surjective, due to the formal smoothness 
\cite[p.111, Cor. 4.6]{Demazure-Gabriel1970}, whose kernel is denoted
by $K_r(O_F)$. 
For any $0<l<r$, let us denote by $K_l(O_F/\frak{p}^r)$ the kernel of
the canonical 
 group homomorphism $G(O_F/\frak{p}^r)\to G(O_F/\frak{p}^l)$ which is
 surjective.  

Through out this papaer, let us 
assume that $p$ is prime to $n$. Then 
the following basic assumptions on $G$ are satisfied 
\begin{itemize}
\item[I)] $B:\frak{g}(\Bbb F)\times\frak{g}(\Bbb F)\to\Bbb F$ is
  non-degenerate, 
\item[II)] for any integers $r=l+l^{\prime}$ with 
           $0<l^{\prime}\leq l$, we have a group isomorphism
$$
 \frak{g}(O_F/\frak{p}^{l^{\prime}})\,\tilde{\to}\,K_l(O_F/\frak{p}^r)
$$
           defined by 
$X\npmod{\frak{p}^{l^{\prime}}}\mapsto1+\varpi^lX\npmod{\frak{p}^r}$,
\item[III)] if $r=2l-1\geq 3$ is odd, then we have a mapping
$$
 \frak{g}(O_F)\to K_{l-1}(O_F/\frak{p}^r)
$$
defined by 
$X\mapsto(1+\varpi^{l-1}X+2^{-1}\varpi^{2l-2}X^2)\npmod{\frak{p}^r}$.
\end{itemize}
The condition I) implies that 
$B:\frak{g}(O_F/\frak{p}^l)\times\frak{g}(O_F/\frak{p}^l)
   \to O_F/\frak{p}^l$ 
is non-degenerate for all $l>0$, and so 
$B:\frak{g}(O_F)\times\frak{g}(O_F)\to O_F$ is also non-degenerate. 
By the condition II), $K_l(O_F/\frak{p}^r)$ is a commutative normal
subgroup of $G(O_F/\frak{p}^r)$, and its character is
$$
 \chi_{\beta}(1+\varpi^lX\npmod{\frak{p}^r})
 =\psi\left(\varpi^{-l^{\prime}}B(X,\beta)\right)
 \quad
 (X\npmod{\frak{p}^{l^{\prime}}}
  \in\frak{g}(O_F/\frak{p}^{l^{\prime}}))
$$
with 
$\beta\npmod{\frak{p}^{l^{\prime}}}
 \in\frak{g}(O_F/\frak{p}^{l^{\prime}})$.

Since 
any finite dimensional complex continuous representation of
the compact group $G(O_F)$ factors through the canonical group
homomorphism $G(O_F)\to G(O_F/\frak{p}^r)$ for some $0<r\in\Bbb Z$, we
want to know the irreducible complex representations of the finite
group $G(O_F/\frak{p}^r)$. 
Let us assume
that $r>1$ and put $r=l+l^{\prime}$ with the minimal integer $l$ such
that $0<l^{\prime}\leq l$, that is
$$
 l^{\prime}=\begin{cases}
             l&:\text{\rm if $r=2l$},\\
             l-1&:\text{\rm if $r=2l-1$}.
            \end{cases}
$$
Let $\delta$ be an irreducible complex representation of
$G(O_F/\frak{p}^r)$. The Clifford's theorem says that the restriction 
$\delta|_{K_l(O_F/\frak{p}^r)}$ is a sum of the 
$G(O_F/\frak{p}^r)$-conjugates of
characters of $K_l(O_F/\frak{p}^r)$:
\begin{equation}
 \delta|_{K_l(O_F/\frak{p}^r)}
 =\left(\bigoplus_{\dot\beta\in\Omega}\chi_{\beta}\right)^m
\label{eq:decomposition-formula-of-delta}
\end{equation}
with an adjoint 
$G(O_F/\frak{p}^{l^{\prime}})$-orbit 
$\Omega\subset\frak{g}(O_F/\frak{p}^{l^{\prime}})$. In this way the
irreducible complex representations of $G(O_F/\frak{p}^r)$ correspond
to adjoint $G(O_F/\frak{p}^{l^{\prime}})$-orbits in 
$\frak{g}(O_F/\frak{p}^{l^{\prime}})$. 

Fix an adjoint $G(O_F/\frak{p}^{l^{\prime}})$-orbit 
$\Omega\subset\frak{g}(O_F/\frak{p}^{l^{\prime}})$ and let us denote
by $\Omega\sphat$ the set of the equivalence classes of the
irreducible complex representations of $G(O_F/\frak{p}^{l^{\prime}})$
correspond to $\Omega$. Then \cite{Takase2021} gives a parametrization
of $\Omega\sphat$ as follows:

\begin{thm}\label{th:parametrization-of-omega-sphat-in-general}
Take a representative $\beta\npmod{\frak{p}^{l^{\prime}}}\in\Omega$
($\beta\in\frak{g}O_F)$) and assume that
\begin{enumerate}
\item the centralizer $G_{\beta}=Z_G(\beta)$ of
      $\beta\in\frak{g}(O_F)$ in $G$ is smooth over $O_F$,
\item the characteristic polynomial 
      $\chi_{\overline\beta}(t)=\det(t\cdot 1_n-\overline\beta)$ of 
      $\overline\beta=\beta\pmod{\frak p}\in\frak{g}(\Bbb F)
        \subset\frak{gl}_n(\Bbb F)$ is
      the minimal polynomial of $\overline\beta\in M_n(\Bbb F)$.
\end{enumerate}
Then there exists a bijection $\theta\mapsto\delta_{\beta,\theta}$ of
the set 
$$
 \left\{\theta\in G_{\beta}(O_F/\frak{p}^r)\sphat\;\;\;
         \text{\rm s.t. $\theta=\chi_{\beta}$ on 
               $G_{\beta}(O_F/\frak{p}^r)\cap K_l(O_F/\frak{p}^r)$}
        \right\}
$$
onto $\Omega\sphat$.
\end{thm}

The correspondence $\theta\mapsto\delta_{\beta,\theta}$ is given by
the following procedure. 
The second condition in the theorem implies 
$$
 G_{\beta}(O_F/\frak{p}^r)
 =G(O_F/\frak{p}^r)\cap
  \left(O_F/\frak{p}^r\right)[\beta\npmod{\frak{p}^r}],
$$
in particular $G_{\beta}(O_F/\frak{p}^r)$ is commutative. So 
$G_{\beta}(O_F/\frak{p}^r)\sphat$ means the character group of 
$G_{\beta}(O_F/\frak{p}^r)$. 

$\Omega\sphat$ consists of the irreducible complex representations
whose restriction to $K_l(O_F/\frak{p}^r)$ contains the character
$\chi_{\beta}$. Then the Clifford's theory says the followings: put
\begin{align*}
 G(O_F/\frak{p}^r;\beta)
 &=\left\{g\in G(O_F/\frak{p}^r)\mid
           \chi_{\beta}(g^{-1}hg)=\chi_{\beta}(h)\;
             \forall h\in K_l(O_F/\frak{p}^r)\right\}\\
 &=\left\{g\in G(O_F/\frak{p}^r)\mid
           \text{\rm Ad}(g)\beta\equiv\beta
               \npmod{\frak{p}^{l^{\prime}}}\right\}
\end{align*}
and let us denote by 
$\text{\rm Irr}(G(O_F/\frak{p}^r;\beta),\chi_{\beta})$ the set of the
equivalence classes of the irreducible
 complex representations $\sigma$ of $G(O_F/\frak{p}^r;\beta)$ such
 that the restriction $\sigma|_{K_l(O_F/\frak{p}^r)}$ contains the
 character $\chi_{\beta}$. 
Then 
$\sigma\mapsto
 \text{\rm Ind}_{G(O_F/\frak{p}^r;\beta)}^{G(O_F/\frak{p}^r)}\sigma$
 gives a bijection of 
$\text{\rm Irr}(G(O_F/\frak{p}^r;\beta),\chi_{\beta})$ onto 
$\Omega\sphat$. 

Since $G_{\beta}$ is smooth over $O_F$, the canonical homomorphism 
$G_{\beta}(O_F/\frak{p}^r)\to G_{\beta}(O_F/\frak{p}^{l^{\prime}})$ is
 surjective. Hence we have
$$
 G(O_F/\frak{p}^r;\beta)
 =G_{\beta}(O_F/\frak{p}^r)\cdot K_{l^{\prime}}(O_F/\frak{p}^r).
$$
If $r=2l$ is even, then $l^{\prime}=l$ and, for any character 
$\theta\in G_{\beta}(O_F/\frak{p}^r)$ such that 
$\theta=\chi_{\beta}$ on 
$G_{\beta}(O_F/\frak{p}^r)\cap K_l(O_F/\frak{p}^r)$, the character 
$$
 \sigma_{\theta,\beta}(gh)=\theta(g)\cdot\chi_{\beta}(h)
 \quad
 (g\in G_{\beta}(O_F/\frak{p}^r), h\in K_l(O_F/\frak{p}^r))
$$
of $G(O_F/\frak{p}^r;\beta)$ is well-defined, and 
$\theta\mapsto\sigma_{\theta,\beta}$ is a surjection onto 
$\text{\rm Irr}(G(O_F/\frak{p}^r;\beta),\chi_{\beta})$. Hence
$$
 \theta\mapsto
 \delta_{\theta,\beta}
 =\text{\rm Ind}_{G(O_F/\frak{p}^r;\beta)}^{G(O_F/\frak{p}^r)}
   \sigma_{\theta,\beta}
$$
is the bijection of Theorem
\ref{th:parametrization-of-omega-sphat-in-general}.

If $r=2l-1$ is odd, then $l^{\prime}=l-1$. Let us denote by 
$\frak{g}_{\beta}=\text{\rm Lie}(G_{\beta})$ the Lie algebra
$O_F$-scheme of the smooth $O_F$-group scheme $G_{\beta}$. Then 
$$
 \Bbb V_{\beta}=\frak{g}(\Bbb F)/\frak{g}_{\beta}(\Bbb F)
$$
is a symplectic $\Bbb F$-space with a symplectic $\Bbb F$-form 
$$
 D_{\beta}(\dot X,\dot Y)=B([X,Y],\overline{\beta})\in\Bbb F
 \quad
 (X,Y\in\frak{g}(\Bbb F)).
$$
Let $H_{\beta}=\Bbb V_{\beta}\times\Bbb C^1$ be the Heisenberg group
associated with $(\Bbb V_{\beta},D_{\beta})$ and 
$(\sigma^{\beta},L^2(\Bbb W^{\prime}))$ the Schr\"odinger
representation of $H_{\beta}$ associated with a polarization 
$\Bbb V_{\beta}=\Bbb W^{\prime}\oplus\Bbb W$. More explicitly the
group operation of $H_{\beta}$ is defined by
$$
 (u,s)\cdot(v,t)=(u+v,st\cdot\widehat{\psi}(2^{-1}D_{\beta}u,v))
$$
where $\widehat{\psi}(\overline x)=\psi(\varpi^{-1}x)$ for 
$\overline x=x\npmod{\frak p}\in\Bbb F$, and the action of 
$h=(u,s)\in H_{\beta}$ on $f\in L^2(\Bbb W^{\prime})$ (a complex-valued
function on $\Bbb W^{\prime}$) is defined by
$$
 (\sigma^{\beta}(h)f)(w)
 =s\cdot\widehat{\chi}
         \left(2^{-1}D_{\beta}(u_-,u_+)+D_{\beta}(w,u_+)\right)\cdot
   f(w+u_-)
$$
where $u=u_-+u_+\in\Bbb V_{\beta}=\Bbb W^{\prime}\oplus\Bbb W$.

Take a character $\theta:G_{\beta}(O_F/\frak{p}^r)\to\Bbb C^{\times}$
such that 
$$
 \theta=\chi_{\beta}\;\;\text{\rm on}\;\;
 G_{\beta}(O_F/\frak{p}^r)\cap K_l(O_F/\frak{p}^r).
$$ 
Then an additive
character $\rho_{\theta}:\frak{g}_{\beta}(\Bbb F)\to\Bbb C^{\times}$
is defined by
$$
 \rho_{\theta}(X\nnpmod{\frak p})
 =\chi\left(-\varpi^{-l}B(X,\beta)\right)\cdot
  \theta\left(1+\varpi^{l-1}X+2^{-1}\varpi^{2l-2}X^2
              \nnpmod{\frak{p}^r}\right)
$$
with $X\in\frak{g}_{\beta}(O_F)$. 
Fix a $\Bbb F$-vector subspace $V\subset\frak{g}(\Bbb F)$ such that 
$\frak{g}(\Bbb F)=V\oplus\frak{g}_{\beta}(\Bbb F)$. Then an
irreducible representation 
$(\sigma^{\beta,\theta},L^2(\Bbb W^{\prime}))$ of
$K_{l-1}(O_F/\frak{p}^r)$ is defined by the following proposition:

\begin{prop}
\label{prop:another-expression-of-pi-beta-psi}
Take a $g=1+\varpi^{l-1}T\npmod{\frak{p}^r}\in K_{l-1}(O_F/\frak{p}^r)$ 
with $T\in\frak{gl}_n(O_F)$. Then we have 
$T\npmod{\frak{p}^{l-1}}\in\frak{g}(O_F/\frak{p}^{l-1})$ and
$$
 \sigma^{\beta,\rho}(g)
  =\tau\left(\varpi^{-l}B(T,\beta)
            -2^{-1}\varpi^{-1}B(T^2,\beta)\right)\cdot
   \rho_{\theta}(Y)\cdot\sigma^{\beta}(v,1)
$$
where $\overline T=[v]+Y\in\frak{g}(\Bbb F)$ with 
$v\in\Bbb V_{\beta}$ and 
$Y\in\frak{g}_{\beta}(\Bbb F)$. 
\end{prop}

Then main result shown in \cite{Takase2021}, under the assumptions of
Theorem \ref{th:parametrization-of-omega-sphat-in-general},
is that there exists a
group homomorphism (not unique) 
$$
 U:G_{\beta}(O_F/\frak{p}^r)\to  GL_{\Bbb C}(L^2(\Bbb W^{\prime}))
$$
such that 
\begin{enumerate}
\item $\sigma^{\beta,\theta}(h^{-1}gh)
       =U(h)^{-1}\circ\sigma^{\beta,\theta}(g)
                              \circ U(h)$ for all 
      $h\in G_{\beta}(O_F/\frak{p}^r)$ and $g\in K_{l-1}(O_F/\frak{p}^r)$,
      and 
\item $U(h)=1$ for all 
      $h\in G_{\beta}(O_F/\frak{p}^r)\cap K_{l-1}(O_F/\frak{p}^r)$.
\end{enumerate}

Now an irreducible representation 
$(\sigma_{\beta,\theta},L^2(\Bbb W^{\prime}))$ is defined by
$$
 \sigma_{\beta,\theta}(hg)
 =\theta(h)\cdot U(h)\circ\sigma^{\beta,\theta}(g)
$$
for 
$hg\in G(O_F/\frak{p}^r;\beta)
 =G_{\beta}(O_F/\frak{p}^r)\cdot K_{l-1}(O_F/\frak{p}^r)$ with 
$h\in G_{\beta}(O_F/\frak{p}^r)$ and $g\in K_{l-1}(O_F/\frak{p}^r)$, and 
$\theta\mapsto\sigma_{\beta,\theta}$ is a surjection onto 
$\text{\rm Irr}(G(O_F/\frak{p}^r;\beta),\chi_{\beta})$. Then 
$$
 \theta\mapsto
 \delta_{\beta,\theta}
 =\text{\rm Ind}_{G(O_F/\frak{p}^r;\beta)}^{G(O_F/\frak{p}^r)}
   \sigma_{\beta,\theta}
$$
is the bijection of Theorem
\ref{th:parametrization-of-omega-sphat-in-general}.

Because the connected $O_F$-group scheme $G=SL_n$ is reductive, 
that is, the fibers $G{\otimes}_{O_F}K$ ($K=F,\Bbb F$) are reductive
$K$-algebraic groups, the dimension of a maximal torus
in $G{\otimes}_{O_F}K$ is independent of $K$ which is denoted by 
$\text{\rm rank}(G)$. For any $\beta\in\frak{g}(O_F)$ we have
\begin{equation}
 \dim_K\frak{g}_{\beta}(K)=\dim\frak{g}_{\beta}{\otimes}_{O_F}K
 \geq\dim G_{\beta}{\otimes}_{O_F}K\geq\text{\rm rank}(G).
\label{eq:dimension-of-centrlizer-on-lie-alg-and-group}
\end{equation}
We say $\beta$ is {\it smoothly regular} over $K$ 
if $\dim_K\frak{g}_{\beta}(K)=\text{\rm rank}(G)$ 
(see \cite[(5.7)]{Springer1966}). 
In this case $G_{\beta}{\otimes}_{O_F}K$ is smooth over $K$. 

In our case of $G=SL_n$, the
following two statements are equivalent for a $\beta\in\frak{g}(O_F)$: 
\begin{enumerate}
\item $\overline\beta\in\frak{g}(K)$ is smoothly regular over $K$,
\item the characteristic polynomial of 
      $\overline\beta\in\frak{g}(K)\subset\frak{gl}_n(K)$ is equal
  to its minimal polynomial
\end{enumerate}
where $\overline\beta\in\frak{g}(K)$ is the image of
$\beta\in\frak{g}(O_F)$ by the canonical morphism
$\frak{g}(O_F)\to\frak{g}(K)$ with $K=F$ or $\Bbb F$.

Since we have canonical isomorphisms 
$$
 \frak{g}(\Bbb F)\,\tilde{\to}\,K_{m-1}(O_F/\frak{p}^m),
 \qquad
 \frak{g}_{\beta}(\Bbb F)\,\tilde{\to}\,
  G_{\beta}(O_F/\frak{p}^m)\cap K_{m-1}(O_F/\frak{p}^m)
$$
and the canonical morphism 
$G_{\beta}(O_F)\to G_{\beta}(O_F/\frak{p}^m)$ is surjective 
for any $m>1$, we have
$$
 |G(O_F/\frak{p}^m)|=|G(\Bbb F)|\cdot q^{(m-1)\dim G},
 \qquad
 |G_{\beta}(O_F/\frak{p}^m)|
 =|G_{\beta}(\Bbb F)|\cdot q^{(m-1)\text{\rm rank}\,G}
$$
for all $m>0$. Then we have
\begin{align*}
 \sharp\Omega\sphat=
 &\sharp\left\{\theta\in G_{\beta}(O_F/\frak{p}^r)\sphat\;\;\;
          \text{\rm s.t. $\theta=\psi_{\beta}$ on 
                $G_{\beta}(O_F/\frak{p}^r)
                  \cap K_l(O_F/\frak{p}^r)$}\right\}\\
 =&\left(G_{\beta}(O_F/\frak{p}^r):G_{\beta}(O_F/\frak{p}^r)
          \cap K_l(O_F/\frak{p}^r)\right)
 =|G_{\beta}(O_F/\frak{p}^l)|\\
 =&|G_{\beta}(\Bbb F)|\cdot q^{(l-1)\text{\rm rank}\,G}
 =\frac{|G(\Bbb F)|}
       {\sharp\overline\Omega}\cdot q^{(l-1)\text{\rm rank}\,G}
\end{align*}
where $\overline\Omega\subset\frak{g}(\Bbb F)$ is the image of 
$\Omega\subset\frak{g}(O_F/\frak{p}^{l^{\prime}})$ under the canonical
morphism $\frak{g}(O_F/\frak{p}^{l^{\prime}})\to\frak{g}(\Bbb F)$. 
On the other hand we have
$$
 \dim\sigma_{\beta,\theta}
 =\begin{cases}
   1&:\text{\rm $r$ is even},\\
   q^{\frac 12\dim_{\Bbb F}(\frak{g}(\Bbb F)/\frak{g}_{\beta}(\Bbb F))}
   =q^{(\dim G-\text{\rm rank}\,G)/2}
     &:\text{\rm $r$ is odd},
  \end{cases}
$$
so we have
\begin{align}
 \dim\delta_{\beta,\theta}
 &=\left(G(O_F/\frak{p}^r):G(O_F/\frak{p}^r;\beta)\right)
    \cdot\dim\sigma_{\beta,\theta} \nonumber\\
 &=\sharp\overline\Omega\cdot q^{(r-2)(\dim G-\text{\rm rank}\,G)/2}.
\label{eq:dimension-formula-of-delta-beta-theta}
\end{align}

\begin{rem}
\label{remark:all-we-need-is-the-surejectivity-of-canonical-hom}
The assumption in Theorem
\ref{th:parametrization-of-omega-sphat-in-general} that the
centralizer $G_{\beta}$ to be smooth $O_F$-group scheme can be
replaced by the surjectivity of the canonical morphisms
$$
 G_{\beta}(O_F)\to G_{\beta}(O_F/\frak{p}^l),
 \quad
 \frak{g}_{\beta}(O_F)\to\frak{g}_{\beta}(O_F/\frak{p}^l),
$$
for all $l>0$. 
\end{rem}

\subsection{Regular irreducible character associated with tamely ramified
            extensions}
\label{subsec:symplectic-space-associated-with-tamely-ramified-ext}
Let $K_/F$ be a field extension of degree $n>1$. 
Let
$$
 e=e(K/F),
 \qquad
 f=f(K/F)
$$
be the ramification index and the inertial degree of $K/F$
respectively so that we have $ef=n$. Since we assume that $p$ is prime
to $n$, the extension $K/F$ is tamely ramified. 

Let $K_0/F$ be the maximal unramified subextension of $K/F$. Then 
$K_0/F$ is a cyclic Galois extension whose Galois group is generated
by the geometric Frobenius automorphism $\text{\rm Fr}$ which induces
the inverse of the Frobenius automorphism $[x\mapsto x^q]$ of the
residue field $\Bbb K_0$ over $\Bbb F$. Since 
$K/K_0$ is totally ramified, there exists a prime element 
$\varpi_K$ of $K$ such that $\varpi_K^e\in K_0$. Then 
$\{1,\varpi_K,\varpi_K^2,\cdots,\varpi_K^{e-1}\}$is an $O_{K_0}$-basis
of $O_K$. The following two propositions are proved by 
Shintani \cite[Lemma 4-7, Cor.1, Cor.2,pp.545-546]{Shintani1968}:

\begin{prop}\label{prop:shintani-lemma-on-generator-of-ok-over-of}
Put 
$\beta=\sum_{i=0}^{e-1}a_i\varpi_K^i\in O_K$ ($a_i\in O_{K_0}$). Then
$O_K=O_F[\beta]$ if and only if the following two conditions are
satisfied:
\begin{enumerate}
\item $a_0^{\text{\rm Fr}}\not\equiv a_0\npmod{\frak{p}_{K_0}}$ if
  $f>1$, 
\item $a_1\in O_{K_0}^{\times}$ if $e>1$.
\end{enumerate}
\end{prop}

\begin{prop}\label{prop:shintani-cor-of-irreducibility-chrara-poly}
Let $\chi_{\beta}(t)\in O_F[t]$ be the characteristic polynomial of
$\beta\in O_K\subset M_n(O_F)$ via the regular representation with
respect to an $O_F$-basis of $O_K$. If $O_K=O_F[\beta]$, then 
\begin{enumerate}
\item $\chi_{\beta}(t)\npmod{\frak{p}_F}\in\Bbb F[t]$ is the minimal
  polynomial of $\overline\beta\in M_n(\Bbb F)$, 
\item $\chi_{\beta}(t)\npmod{\frak{p}_F}=p(t)^e$ with an irreducible
  polynomial $p(t)\in\Bbb F[t]$,
\item if $e>1$, then $\chi_{\beta}(t)\npmod{\frak{p}_F^2}$ is
  irreducible over $O_F/\frak{p}_F^2$.
\end{enumerate}
\end{prop}

We can prove the following

\begin{prop}\label{prop:regular-element-of-gl(n)}
Take a $\beta\in M_n(O_F)$ whose the characteristic polynomial be
$$
 \chi_{\beta}(t)=t^n-a_nt^{n-1}-\cdots-a_2t-a_1.
$$
If $\chi_{\beta}(t)\npmod{\frak{p}_F}\in\Bbb F[t]$ is the minimal
polynomial of $\beta\npmod{\frak{p}_F}\in M_n(\Bbb F)$, then
\begin{enumerate}
\item $\{X\in M_n(O_F)\mid[X,\beta]=0\}=O_F[\beta]$,
\item for any $m>0$, put 
$\overline\beta=\npmod{\frak{p}_F^m}\in M_n(O_F/\frak{p}_F^m)$, then 
$$
 \left\{X\in M_n(O_F/\frak{p}_F^m)\mid
               [X,\overline\beta]=0\right\}
 =O_F/\frak{p}_F^m[\overline\beta],
$$
\item there exists a $g\in GL_n(O_F)$ such that
$$
 g\beta g^{-1}=\begin{bmatrix}
                0&0&\cdots&0&a_1\\
                1&0&\cdots&0&a_2\\
                 &1&\ddots&\vdots&\vdots\\
                 & &\ddots&0&a_{n-1}\\
                 & &      &1&a_n
               \end{bmatrix}.
$$
\end{enumerate}
\end{prop}

Now our $O_F$-group scheme $G=SL_n$ is defined based on the free
$O_F$-module $O_K$ with a fixed $O_F$-basis of $O_K$. 
Take a $\beta\in O_K$ such that 
$O_K=O_F[\beta]$ and $T_{K/F}(\beta)=0$. 
Identify the $F$-algebra $K$ with a 
$F$-subalgebra of $\text{\rm End}_n(F)$ by the regular
representation with respect to the fixed $O_F$-basis of $O_K$. 
By Proposition \ref{prop:shintani-cor-of-irreducibility-chrara-poly},
the characteristic polynomial of 
$\overline\beta=\beta\npmod{\frak{p}_F}\in M_n(\Bbb F)$ is equal to
its minimal polynomial. Then, by Proposition 
\ref{prop:regular-element-of-gl(n)}, we have
$$
 \{X\in M_n(O_F)\mid[X,\beta]=0\}=O_F[\beta]=O_K
$$
and
$$
 \{X\in M_n(O_F/\frak{p}_F^l)\mid[X,\overline\beta]=0\}
 =O_F/\frak{p}_F^l[\overline\beta]
 =O_K/\frak{p}_K^{el}
$$
for any $m>0$. Put 
$$
 U_{K/F}
 =\{\varepsilon\in O_K^{\times}\mid N_{K/F}(\varepsilon)=1\}.
$$
Then we have
$$
 G_{\beta}(O_F)=G(O_F)\cap O_K=U_{K/F}.
$$
We have also
$$
 \frak{g}_{\beta}(O_F)=\frak{g}(O_F)\cap O_K
 =\{X\in O_K\mid T_{K/F}(X)=0\}
$$
and
\begin{align*}
 G_{\beta}(O_F/\frak{p}_F^l)
 &=\{\overline\varepsilon\in\left(O_K/\frak{p}_K^{el}\right)^{\times}
    \mid N_{K/F}(\varepsilon)\equiv
    1\npmod{\frak{p}_F^l}\},\\
 \frak{g}_{\beta}(O_F/\frak{p}_F^l)
 &=\{\overline X\in O_K/\frak{p}_K^{el}\mid
     T_{K/F}(X)\equiv 0\npmod{\frak{p}_F^l}\}
\end{align*}
for all $l>0$. Then the canonical morphisms
$$
 G_{\beta}(O_F)\to G_{\beta}(O_F/\frak{p}_F^l),
 \qquad
 \frak{g}_{\beta}(O_F)\to\frak{g}_{\beta}(O_F/\frak{p}_F^l)
$$
are surjective for all $l>0$. In fact, take an 
$\varepsilon\in O_K^{\times}$ such that 
$N_{K/F}(\varepsilon)\equiv 0\npmod{\frak{p}_F^l}$. Since $K/F$ is
tamely ramified, we have $N_{K/F}(1+\frak{p}_K^{el})=1+\frak{p}_F^l$,
that is, there exists $\eta\in 1+\frak{p}_K^{el}$ such that 
$N_{K/F}(\eta)=N_{K/F}(\varepsilon)$. Then 
$\xi=\eta^{-1}\varepsilon\in U_{K/F}$ such that 
$\xi\equiv\varepsilon\npmod{\frak{p}_K^{el}}$. Take a $x\in O_K$ such
that $T_{K/F}(x)\equiv 0\npmod{\frak{p}_F^l}$, or 
$T_{K/F}(x)=a\cdot\varpi_F^l$ with $a\in O_F$. Since $K/F$ is tamely
ramified, we have $T_{K/F}(O_K)=O_F$, that is, there exists a 
$y\in O_K$ such that $T_{K/F}(y)=a$. Then $z=x-y\varpi_F^l\in O_K$
such that $T_{K/F}(z)=0$ and $z\equiv x\npmod{\frak{p}_K^{el}}$. 

Due to Remark
\ref{remark:all-we-need-is-the-surejectivity-of-canonical-hom}, 
we can apply the general theory of subsection 
\ref{subsec:regular-irred-character-of-hyperspecial-compact-subgroup}
to our $\beta\in\frak{g}(O_F)$. Take an integer $r>1$ and put 
$r=l+l^{\prime}$ with minimal integer $l$ such that 
$0<l^{\prime}\leq l$. 
Let $\Omega\subset\frak{g}(O_F/\frak{p}_F^{l^{\prime}})$ be the 
adjoint $G(O_F/\frak{p}_F^{l^{\prime}})$-orbit of 
$\beta\npmod{\frak{p}_F^{l^{\prime}}}
 \in\frak{g}(O_F/\frak{p}_F^{l^{\prime}})$, and $\Omega\sphat$ the set
 of the equivalent classes of the irreducible representations of 
$G(O_F/\frak{p}_F^r)$ corresponding to $\Omega$ via Clifford's theory
 described in subsection
 \ref{subsec:regular-irred-character-of-hyperspecial-compact-subgroup}. 
Then we have a bijection $\theta\mapsto\delta_{\beta,\theta}$ 
of the continuous unitary character $\theta$ of $U_{K/F}$ such that
\begin{enumerate}
\item $\theta$ factors through the canonical morphism
         $U_{K/F}\to\left(O_K/\frak{p}_K^{er}\right)^{\times}$, 
\item for an $\varepsilon\in U_{K/F}$ such that 
      $\varepsilon\equiv 1+\varpi_F^lx\npmod{\frak{p}_K^{er}}$ with 
      $x\in O_K$ and 
      $T_{K/F}(x)\equiv 0\npmod{\frak{p}_F^{l^{\prime}}}$, we
  have 
$\theta(\varepsilon)
 =\psi\left(\varpi_F^{-l^{\prime}}T_{K/F}(x\beta)\right)$.
\end{enumerate}
onto $\Omega\sphat$. Here $\psi:F\to\Bbb C^{\times}$ is a continuous
unitary character of the additive group $F$ such that 
$\{x\in F\mid\psi(xO_F)=1\}=O_F$. Then we have

\begin{prop}\label{prop:dimension-of-delta-beta-theta}
$$
 \dim\delta_{\beta,\theta}
 =\frac{q^{r\cdot n(n-1)/2}}
          {1-q^{-f}}\cdot
  \frac 1{(O_F^{\times}:N_{K/F}(O_K^{\times}))}\cdot
  \prod_{k=1}^n\left(1-k^{-k}\right).
$$
\end{prop}
\begin{proof}
For the dimension formula
\eqref{eq:dimension-formula-of-delta-beta-theta}, we have
$$
 \dim G=n^2-1,
 \quad
 \text{\rm rank}\,G=n-1,
 \quad
 \sharp\overline\Omega=\frac{|G(\Bbb F)|}
                            {|G_{\beta}(\Bbb F)|}
$$
and
$$
 |G(\Bbb F)|=|SL_n(\Bbb F)|
 =q^{n^2-1}\prod_{k=2}^n\left(1-q^{-k}\right).
$$
On the other hand $G_{\beta}(\Bbb F)$ is the kernel of
$$
 (\ast):\left(O_K/\frak{p}_K\right)^{\times}\to
        \left(O_F/\frak{p}_F\right)^{\times}
 \quad
 \left(\varepsilon\npmod{\frak{p}_K^e}\mapsto
       {N_{K/F}(\varepsilon)}\npmod{\frak{p}_F}\right).
$$
Since $K/F$ is tamely ramified extension, we have
$$
 1+\frak{p}_F=N_{K/F}(1+\frak{p}_K^e)
 \subset N_{K/F}(O_K^{\times})\subset O_F^{\times},
$$
hence
\begin{align*}
 |G_{\beta}(\Bbb F)|
 &=\frac{\left|\left(O_K/\frak{p}_K^e\right)^{\times}\right|
         (O_F^{\times}:N_{K/F}(O_K^{\times}))}
        {\left|\left(O_F/\frak{p}_F\right)^{\times}
         \right|}
  =(O_F^{\times}:N_{K/F}(O_K^{\times}))\cdot
   \frac{q^{fe}-q^{f(e-1)}}
        {q-1}\\
 &=(O_F^{\times}:N_{K/F}(O_K^{\times}))\cdot q^{n-1}\cdot
   \frac{1-q^{-f}}
        {1-q^{-1}}.
\end{align*}
\end{proof}

\subsection{Construction of supercuspidal representations}
\label{subsec:construction-of-supercuspidal-representataion}
We will keep the notations of the preceding subsection. 
The purpose of this subsection is to prove the following theorem:

\begin{thm}\label{th:supercuspidal-representation-of-sl(n)}
If 
$l^{\prime}=\left\lfloor\frac r2\right\rfloor\geq 2(e-1)$, 
then the compactly induced representation 
$\pi_{\beta,\theta}
 =\text{\rm ind}_{G(O_F)}^{G(F)}\delta_{\beta,\theta}$ is an
 irreducible supercuspidal representation of $G(F)=SL_n(F)$ such
 that
\begin{enumerate}
\item the multiplicity of $\delta_{\beta,\theta}$ in 
      $\pi_{\beta,\theta}|_{G(O_F)}$ is one, 
\item $\delta_{\beta,\theta}$ is the unique irreducible unitary
  constituent of $\pi_{\beta,\theta}|_{G(O_F)}$ which factors through the
  canonical morphism $G(O_F)\to G(O_F/\frak{p}^r)$,
\item with respect to the Haar measure on $G(F)$ such that the volume
  of $G(O_F)$  is one, the formal degree of $\pi_{\beta,\theta}$ is
    equal to
$$
    \dim\delta_{\beta,\theta}
    =\frac{q^{r\cdot n(n-1)/2}}
          {1-q^{-f}}\cdot
  \frac 1{(O_F^{\times}:N_{K/F}(O_K^{\times}))}\cdot
  \prod_{k=1}^n\left(1-k^{-k}\right).
$$
\end{enumerate}
\end{thm}

The rest of this subsection is devoted to the proof. 

We have the Cartan decomposition 
\begin{equation}
 G(F)=\bigsqcup_{m\in\Bbb M}G(O_F)\varpi_F^mG(O_F)
\label{eq:cartan-decomposition-of-sl(n)}
\end{equation}
where 
$$
 \Bbb M
 =\left\{m=(m_1,m_2,\cdots,m_n)\in\Bbb Z^n\biggm|
   \begin{array}{l}
    m_1\geq m_2\geq\cdots\geq m_n,\\
    m_1+m_2+\cdots+m_n=0
   \end{array}\right\}
$$
and
$$
 \varpi_F^m=\begin{bmatrix}
           \varpi_F^{m_1}&      &            \\
                       &\ddots&            \\
                       &      &\varpi_F^{m_n}
          \end{bmatrix}
 \;\text{\rm for}\;
 m=(m_1,\cdots,m_n)\in\Bbb M.
$$
For an integer $1\leq i<n$, let
$$
 U_i=\left\{\begin{bmatrix}
             1_i&B\\
             0&1_{n-i}
            \end{bmatrix}\biggm| B\in M_{i,n-i}\right\}
$$
be the unipotent part of the parabolic subgroup
$$
 P_i=\left\{\begin{bmatrix}
             A&B\\
             0&D
            \end{bmatrix}\in G\biggm| A\in GL_i, D\in GL_{n-i}
           \right\}.
$$
Put $U_i(\frak{p}_F^a)=U_i(O_F)\cap K_a(O_F)$ for a positive integer
$a$. 
 
\begin{prop}\label{prop:admissibility-of-compactly-induced-rep}
If $K/F$ is unramified or $r\geq 4$, then 
$\text{\rm ind}_{G(O_F)}^{G(F)}\delta_{\beta,\theta}$ is admissible representation of
$G(F)$. 
\end{prop}
\begin{proof}
We will
prove that the dimension of the space of the $K_a(O_F)$-fixed 
vectors is finite for any integer $a>0$. The Cartan decomposition 
\eqref{eq:cartan-decomposition-of-sl(n)} gives
$$
 G(F)=\bigsqcup_{s\in S}K_a(O_F)sG(O_F)
$$
with 
$$
 S=\{k\varpi_F^m\mid \dot k\in K_a(O_F)\backslash G(O_F), 
                   m\in\Bbb M\}.
$$
Then we have
$$
 \left.\text{\rm ind}_{G(O_F)}^{G(F)}\delta_{\beta,\theta}\right|_{K_a(O_F)}
 =\bigoplus_{s\in S}
  \text{\rm ind}_{K_a(O_F)\cap sG(O_F)s^{-1}}^{K_a(O_F)}
   \delta_{\beta,\theta}^s
$$
with $\delta_{\beta,\theta}^s(h)=\delta_{\beta,\theta}(s^{-1}hs)$ 
($h\in K_a(O_F)\cap sG(O_F)s^{-1}$). 
The Frobenius reciprocity gives
$$
 \text{\rm Hom}_{K_a(O_F)}\left(\text{\bf 1},
    \text{\rm ind}_{G(O_F)}^{G(F)}\delta_{\beta,\theta}\right)
  =\bigoplus_{s\in S}\text{\rm Hom}_{s^{-1}K_a(O_F)s\cap G(O_F)}
      (\text{\bf 1},\delta_{\beta,\theta}).
$$
Here $\text{\bf 1}$ is the one-dimensional trivial representation of 
$K_a(O_F)$. If 
$$
 \text{\rm Hom}_{K_a(O_F)}\left(\text{\bf 1},
  \text{\rm ind}_{G(O_F)}^{G(F)}\delta_{\beta,\theta}\right)\neq 0
$$
then there exists a 
$$
 s=k\varpi_F^m\in S
 \quad
 (k\in G(O_F), m=(m_1,\cdots,m_n)\in\Bbb M)
$$ 
such that 
$\text{\rm Hom}_{s^{-1}K_a(O_F)s\cap G(O_F)}(\text{\bf 1},
  \delta_{\beta,\theta})\neq 0$. 
If 
$$
 \text{\rm Max}\{m_i-m_{i+1}\mid 1\leq i<n\}=m_i-m_{i+1}\geq a
$$
then $\varpi_F^m U_i(O_F)\varpi_F^{-m}\subset K_a(O_F)$. 
So we have $U_i(O_F)\subset s^{-1}K_a(O_F)s\cap G(O_F)$ so that
$$
 \text{\rm Hom}_{U_i(\frak{p}^l)}(\text{\bf 1},\delta_{\beta,\theta})
 \supset
 \text{\rm Hom}_{s^{-1}K_a(O_F)s\cap G(O_F)}
   (\text{\bf 1},\delta_{\beta,\theta})\neq 0
$$
where $U_i(\frak{p}^l)=U_i(O_F)\cap G(\frak{p}^l)$. Then the
decomposition \eqref{eq:decomposition-formula-of-delta} implies that
there exists a $g\in G(O_F)$ such that 
$\chi_{\text{\rm Ad}(g)\beta}(h)=1$ for all $h\in U_i(\frak{p}^l)$,
that is
$$
 \tau\left(\varpi_F^{-l^{\prime}}
   \text{\rm tr}(g\beta g^{-1}\begin{bmatrix}
                               0&B\\
                               0&0
                              \end{bmatrix})\right)=0
$$
for all $B\in M_{i,n-1}(O_F)$. This means 
$$
 g\beta g^{-1}\equiv\begin{bmatrix}
                     A&\ast\\
                     0&D
                    \end{bmatrix}\npmod{\frak{p}^{l^{\prime}}}
$$
with $A\in M_i(O_F), D\in M_{n-i}(O_F)$, that is
$$
 \chi_{\beta}(t)\equiv\det(t1_i-A)\cdot\det(t1_{n-i}-D)
  \npmod{\frak{p}^{l^{\prime}}}.
$$
If $K/F$ is unramified, this is contradict against 2) of Proposition 
\ref{prop:shintani-cor-of-irreducibility-chrara-poly}. If
$K/F$ is ramified, then $r\geq 4$ and $l^{\prime}\geq 2$ and a
contradiction to 3) of the proposition. So we have 
$$
 \text{\rm Max}\{m_i-m_{i+1}\mid 1\leq i<n\}<a.
$$
This implies that the number of $s\in S$ such that 
$\text{\rm Hom}_{s^{-1}K_a(O_F)s\cap G(O_F)}
  (\text{\bf 1},\delta_{\beta,\theta})\neq 0$ is finite, and then
$$
 \dim_{\Bbb C}\text{\rm Hom}_{K_a(O_F)}\left(
   \text{\bf 1},\text{\rm ind}_{G(O_F)}^{G(F)}\delta_{\beta,\theta}\right)
 <\infty.
$$
\end{proof}

\begin{prop}\label{prop:minimal-k-type-of-induced-rep}
If $l^{\prime}=\left\lfloor\frac r2\right\rfloor\geq 2(e-1)$, then
\begin{enumerate}
\item $\dim_{\Bbb C}\text{\rm Hom}_{G(O_F)}\left(
        \delta_{\beta,\theta},
        \text{\rm ind}_{G(O_F)}^{G(F)}\delta_{\beta,\theta}\right)=1$,
\item if $\delta$ is an irreducible representation of $G(O_F)$ which
  factors through the canonical surjection $G(O_F)\to
  G(O_F/\frak{p}^r)$ such that
$$
 \text{\rm Hom}_{G(O_F)}\left(
   \delta,\text{\rm ind}_{G(O_F)}^{G(F)}\delta_{\beta,\theta}\right)\neq 0,
$$
  then $\delta=\delta_{\beta,\theta}$.
\end{enumerate}
\end{prop}
\begin{proof}
Cartan decomposition \eqref{eq:cartan-decomposition-of-sl(n)} gives
$$
 \left.\text{\rm ind}_{G(O_F)}^{G(F)}\delta_{\beta,\theta}\right|_{G(O_F)}
 =\bigoplus_{m\in\Bbb M}
   \text{\rm ind}_{G(O_F)\cap\varpi^mG(O_F)\varpi^{-m}}^{G(O_F)}
    \delta_{\beta,\theta}^{\varpi^m}.
$$
Then Frobenius reciprocity gives
$$
 \text{\rm Hom}_{G(O_F)}\left(\delta,
  \text{\rm ind}_{G(O_F)}^{G(F)}\delta_{\beta,\theta}\right)
 =\bigoplus_{m\in\Bbb M}
   \text{\rm Hom}_{\varpi^{-m}G(O_F)\varpi^m\cap G(O_F)}
    \left(\delta^{\varpi^{-m}},\delta_{\beta,\theta}\right).
$$
Now take a $m=(m_1,\cdots,m_n)\in\Bbb M$ such that
$$
 \text{\rm Hom}_{\varpi^{-m}G(O_F)\varpi^m\cap G(O_F)}
    \left(\delta^{\varpi^{-m}},\delta_{\beta,\theta}\right)\neq 0.
$$
Suppose 
$$
 \text{\rm Max}\{m_k-m_{k+1}\mid 1\leq k<n\}=m_i-m_{i+1}\geq a
$$
with a integer $0<a\leq l^{\prime}$. Then 
$$
 U_i(O_F)=\varpi^{-m}U_i(O_F)\varpi^m\cap U_i(O_F)
 \subset\varpi^{-m}G(O_F)\varpi^m\cap G(O_F)
$$
and we have
$$
 \text{\rm Hom}_{U_i(\frak{p}^{r-a})}
  \left(\delta^{\varpi^{-m}},\delta_{\beta,\theta}\right)
 \supset
 \text{\rm Hom}_{\varpi^{-m}G(O_F)\varpi^m\cap G(O_F)}
  \left(\delta^{\varpi^{-m}},\delta_{\beta,\theta}\right)\neq 0.
$$
Since 
$\varpi^mU_i(\frak{p}^{r-a})\varpi^{-m}\subset U_i(\frak{p}^r)
 \subset\text{\rm Ker}\,\delta$, we have 
$$
 \text{\rm Hom}_{U_i(\frak{p}^{r-a})}(\text{\bf 1},\delta_{\beta,\theta})\neq 0.
$$
Here $\text{\bf 1}$ is the trivial one-dimensional representation of
$U_i(\frak{p}^{r-a})$. Since $r-a\geq l$ and hence 
$U_i(\frak{p}^{r-a})\subset U_i(\frak{p}^l)$, there exists a 
$g\in G(O_F)$ such that  
$\psi_{\text{\rm Ad}(g)\beta}(h)=1$ for all 
$h\in U_i(\frak{p}^{r-a})$, due to the decomposition 
\eqref{eq:decomposition-formula-of-delta}. 
This means 
$$
 g\beta g^{-1}\equiv\begin{bmatrix}
                     A&\ast\\
                     0&D
                    \end{bmatrix}\npmod{\frak{p}^a}
$$
with some $A\in M_i(O_F)$ and $D\in M_{n-i}(O_F)$. Then
\begin{equation}
 \chi_{\beta}(t)\equiv
  \det(t1_i-A)\cdot\det(t1_{n-i}-D)\npmod{\frak{p}^a}.
\label{eq:decomposition-of-charatcteristic-polynomial-k-type}
\end{equation}
If $a\geq 2$, this decomposition of the characteristic polynomial 
contradicts to Proposition
\ref{prop:shintani-cor-of-irreducibility-chrara-poly}. 
So $a=1$. Then 2) of  Proposition
\ref{prop:shintani-cor-of-irreducibility-chrara-poly} implies that 
$i=\deg\det(t1_i-A)$ is a multiple of $f$. So we have
$m_1-m_n<e$. Note that we have 
$$
 \varpi^m K_{l+m_1-m_n}(O_F)\varpi^{-m}
 \subset K_l(O_F).
$$
So if $\delta$ corresponds, as explained in subsection 
\ref{subsec:regular-irred-character-of-hyperspecial-compact-subgroup}, 
to an adjoint $G(O_F/\frak{p}_F^{l^{\prime}})$-orbit
$\Omega^{\prime}\subset\frak{g}(O_F/\frak{p}_F^{l^{\prime}})$ of 
$\gamma\npmod{\frak{p}_F^{l^{\prime}}}$
($\gamma\in\frak{g}(O_F)$), then there exist $g,h\in G(O_F)$
such that 
$$
 \chi_{\text{\rm Ad}(g)\beta}(x)
 =\chi_{\text{\rm Ad}(h)\gamma}(\varpi^mx\varpi^{-m})
$$
for all $x\in K_{l+m_1-m_n}(O_F)$. This means 
$$
 g\beta g^{-1}-\varpi^{-m}h\gamma h^{-1}\varpi^m
 \in M_n(\frak{p}^{l^{\prime}-(m_1-m_n)}).
$$
Since $2(m_1-m_n)\leq 2(e-1)\leq l^{\prime}$, we have
$$
 \varpi^mg\beta g^{-1}\varpi^{-m}
 \in M_n(\frak{p}^{l^{\prime}-2(m_1-m_n)})
 \subset M_n(O_F).
$$
Then there exists a $g^{\prime}\in GL_n(O_F)$ such that 
$\varpi^mg\beta g^{-1}\varpi^{-m}=g^{\prime}\beta g^{\prime -1}$ and
hence $g^{\prime -1}\varpi^mg\in K$ due Proposition 
\ref{prop:regular-element-of-gl(n)}.
On the other hand we have 
$$
 N_{K/F}(g^{\prime -1}\varpi^mg)
 =\det(g^{\prime -1}\varpi^mg)\in O_F^{\times}
$$
so that 
$g^{\prime -1}\varpi^mg\in O_K^{\times}\subset GL_n(O_F)$. Hence 
$m=(0,\cdots,0)$. Now we have proved
$$
 \text{\rm Hom}_{G(O_F)}
  \left(\delta,\text{\rm ind}_{G(O_F)}^{G(F)}\delta_{\beta,\theta}\right)
 =\text{\rm Hom}_{G(O_F)}(\delta,\delta_{\beta,\theta})
$$
which implies the two statements of the proposition.
\end{proof}

The admissible representation 
$\pi_{\beta,\theta}
 =\text{\rm ind}_{G(O_F)}^{G(F)}\delta_{\beta,\theta}$ of $G(F)$ is
 irreducible. In fact, if there exists a $G(F)$-subspace
$0\lvertneqq W\lvertneqq
 \text{\rm ind}_{G(O_F)}^{G(F)}\delta_{\beta,\theta}$, we have
\begin{align*}
 0\neq\text{\rm Hom}_{G(F)}
    (W,\text{\rm ind}_{G(O_F)}^{G(F)}\delta_{\beta,\theta})
  &\subset\text{\rm Hom}_{G(O_F)}
    (W,\text{\rm Ind}_{G(O_F)}^{G(F)}\delta_{\beta,\theta})\\
  &=\text{\rm Hom}_{G(O_F)}(W,\delta)
\end{align*}
by Frobenius reciprocity. Hence 
$\delta\hookrightarrow W|_{G(O_F)}$. On the other hand, we have
\begin{align*}
 0&\neq
    \text{\rm Hom}_{G(F)}
      \left(\text{\rm ind}_{G(O_F)}^{G(F)}\delta_{\beta,\theta},
      \left(\text{\rm ind}_{G(O_F)}^{G(F)}\delta_{\beta,\theta}
             \right)/W\right)\\
 &=\text{\rm Hom}_{G(O_F)}\left(\delta_{\beta,\theta},
     \left(\text{\rm ind}_{G(O_F)}^{G(F)}\delta_{\beta,\theta}
           \right)/W\right),
\end{align*}
hence 
$\delta\hookrightarrow
 \left(\text{\rm ind}_{G(O_F)}^{G(F)}\delta_{\beta,\theta}\right)/W$. 
Now $\text{\rm ind}_{G(O_F)}^{G(F)}\delta_{\beta,\theta}$ is
semi-simple $G(O_F)$-module, we have
$$
 \dim_{\Bbb C}
 \text{\rm Hom}_{G(O_F)}(\delta_{\beta,\theta},
         \text{\rm ind}_{G(O_F)}^{G(F)}\delta_{\beta,\theta})
 \geq 2
$$
which contradicts to the first statement of Proposition 
\ref{prop:minimal-k-type-of-induced-rep}.

Then $\pi_{\beta,\theta}$ is a
supercuspidal representation of $G(F)$ whose formal degree with
respect to the Haar measure $d_{G(F)}(x)$of $G(F)$ such that 
$\int_{G(O_F)}d_{G(F)}(x)=1$ is equal to $\dim\delta_{\beta,\theta}$. 
We have completed the proof of Theorem
\ref{th:supercuspidal-representation-of-sl(n)}.

\section{Kaleta's $L$-parameter}
\label{sec:kaleta-l-parameter}

\subsection{Local Langlands correspondence of elliptic tori}
\label{subsec:local-langlands-correspondence-of-elliptic-tori}
Let $K_+/F$ be 
a finite extension, $K/K_+$ a quadratic extension with a non-trivial
element $\tau$ of $\text{\rm Gal}(K/K_+)$. Let us denote by $L$
an arbitrary Galois extension over $F$ containing $K$ for which let us 
denote by  
$$
 \text{\rm Emb}_F(K,L)
 =\left\{\sigma|_K\mid\sigma\in\text{\rm Gal}(L/F)\right\}
$$
the set of the embeddings over $F$ of $K$ into $L$.

Let us denote by $\Bbb V$ the
$\overline F$-algebra of the functions $v$ on 
$\text{\rm Emb}_F(K,\overline F)$ with values in $\overline F$. 
The action of 
$\sigma\in\text{\rm Gal}(\overline F/F)$ on 
$v\in\Bbb V$ is defined by 
$v^{\sigma}(\gamma)=v(\gamma\sigma^{-1})^{\sigma}$. Then fixed point
subspace ${\Bbb V}^{\text{\rm Gal}(\overline F/L)}=\Bbb V(L)$ 
is the set of the
functions on $\text{\rm Emb}_F(K,L)$ with values in $L$, and 
${\Bbb V}^{\text{\rm Gal}(\overline F/F)}=\Bbb V(F)$ is identified
with $K$ via $v\mapsto v(\text{\bf 1}_K)$. 

The action of $\sigma\in\text{\rm Gal}(\overline F/F)$ on 
$g\in SL_{\overline F}(\Bbb V)$ is defined by 
$v\cdot g^{\sigma}=(v^{\sigma^{-1}}\cdot g)^{\sigma}$. Then 
the fixed point subgroup 
$SL_{\overline F}(\Bbb V)^{\text{\rm Gal}(\overline F/F)}$ is identified with 
$SL_F(K)$ via $g\mapsto g|_K$.

Put $S=\text{\rm Res}_{K/F}\Bbb G_m$ which is identified with the
multiplicative group $\Bbb V^{\times}$. Then $S(F)$ is identified with
the multiplicative group $K^{\times}$.

The group $X(S)$ of the characters over $\overline F$ of $S$ 
is a free $\Bbb Z$-module with $\Bbb Z$-basis 
$\{b_{\delta}\}_{\delta\in\text{\rm Emb}_F(K,\overline F)}$ where 
$b_{\delta}(s)=s(\delta)$ for $s\in S$. The dual torus 
$S\sphat=X(S){\otimes}_{\Bbb Z}\Bbb C^{\times}$ is identified with the
group of the functions $s$ on $\text{\rm Emb}_F(K,\overline F)$ with
values in $\Bbb C^{\times}$. 
The action of $\sigma\in W_F\subset\text{\rm Gal}(\overline F/F)$ on
$S$ induces the action on $X(S)$ such that 
$b_{\delta}^{\sigma}=b_{\delta\sigma}$, and hence the
action on $s\in S\sphat$ is defined by 
$s^{\sigma}(\gamma)=s(\gamma\sigma^{-1})$. 

Since we have a bijection $\dot\rho\mapsto\rho|_K$ of 
$W_K\backslash W_F$ onto $\text{\rm Emb}_F(K,\overline F)$, 
the $\overline F$-algebra $\Bbb V$ (resp. the torus $S$, $S\sphat$) is
identified 
with the set of the left$W_K$-invariant functions on $W_F$ with values
in $\Bbb \overline\overline F$ 
(resp. $\overline F^{\times}$, $\Bbb C^{\times}$). 

The local Langlands correspondence for the torus $S$ is the isomorphism 
\begin{equation}
 H^1(W_F,S\sphat)\,\tilde{\to}\,
 \text{\rm Hom}(W_K,\Bbb C^{\times})
\label{eq:local-langlands-correspondence-for-torus}
\end{equation}
given by 
$[\alpha]\mapsto[\rho\mapsto\alpha(\rho)(\text{\bf 1}_K)]$. The
inverse mapping is defined as follows. Let
$$
 l:\text{\rm Emb}_F(K,\overline F)\to W_F
$$
be a section of the restriction mapping 
$W_F\to\text{\rm Emb}_F(K,\overline F)$, that is 
$l(\gamma)|_K=\gamma$ for all 
$\gamma\in\text{\rm Emb}_F(K,\overline F)$ and  
$l(\text{\bf 1}_K)=1$, and put
$$
 J(\gamma,\sigma)=l(\gamma)\sigma l(\gamma\sigma)^{-1}\in W_K
 \;\;\text{\rm for $\gamma\in\text{\rm Emb}_F(K,\overline F)$, 
           $\sigma\in W_F$.}
$$
Take a $\psi\in\text{\rm Hom}(W_K,\Bbb C^{\times})$ and define 
$\alpha\in Z^1(W_F,S\sphat)$ by
$$
 \alpha(\sigma)(\rho)
 =\alpha(\sigma\rho^{-1})(1)\cdot\alpha(\rho^{-1})(1)
 \;\;\text{\rm with}\;\;
 \alpha(\sigma)(1)
 =\psi\left(J(\text{\bf 1}_K,\sigma^{-1})^{-1}\right)
$$
for all $\sigma,\rho\in W_F$. Then $\psi\mapsto[\alpha]$ is the
inverse mapping of the isomorphism 
\eqref{eq:local-langlands-correspondence-for-torus}. 

If we restrict the isomorphism 
\eqref{eq:local-langlands-correspondence-for-torus} 
to continuous group homomorphisms, we have an isomorphism 
\begin{equation}
 H^1_{\text{\rm conti}}(W_F,S\sphat)\,\tilde{\to}\,
 \text{\rm Hom}_{\text{\rm conti}}(K^{\times},\Bbb C^{\times})
\label{eq:continuous-local-langlands-correspondence-for-torus}
\end{equation}
via \eqref{eq:local-langlands-correspondence-for-torus} combined with
the isomorphism of the local class filed theory
$$
 \delta_K:K^{\times}\,\tilde{\to}\,
          W_K/\overline{[W_K,W_K]}.
$$

Let $T$ be a subtorus of $S$ wich is identified with 
the multiplicative subgroup of $\Bbb V^{\times}$
consisting of the functions $s$ on 
$\text{\rm Emb}_F(K,\overline F)$ to 
${\overline F}^{\times}$
such that 
$$
 \prod_{\gamma\in\text{\rm Emb}_F(K,\overline F)}s(\gamma)=1.
$$ 
In other words $T$ is 
a maximal torus of $SL_{\overline F}(\Bbb V)$ by identifying $s\in T$ with 
$[v\mapsto v\cdot s]\in SL_{\overline F}(\Bbb V)$. 
The fixed point subgroup 
$T^{\text{\rm Gal}(\overline F/F)}=T(F)$ is identified with 
$$
 U_{K/F}
 =\{\varepsilon\in O_K^{\times}\mid N_{K/F}(\varepsilon)=1\}
 \quad
 \text{\rm by $s\mapsto s(\text{\bf 1}_K)$. }
$$
The restriction mapping gives  a canonical surjection
\begin{equation}
 \text{\rm Hom}_{\text{\rm conti.}}(K^{\times},\Bbb C^{\times})
 \to
 \text{\rm Hom}_{\text{\rm conti.}}(U_{K/F},\Bbb C^{\times}).
\label{eq:canonical-surjection-of-hom-group-of-elliptic-tori}
\end{equation}
The restriction from $S$ to $T$ gives a surjection $X(S)\to X(T)$
whose kernel is the subgroup of $X(S)$ generated by 
$\sum_{\gamma\in\text{\rm Emb}_F(K,\overline F)}b_{\gamma}$. 
Then the dual torus is 
$$
 T\sphat=X(T){\otimes}_{\Bbb Z}\Bbb C^{\times}
        =S\sphat/\Bbb C^{\times}
$$
where $\Bbb C^{\times}\subset S\sphat$ is the subgroup of the 
$\Bbb C^{\times}$-valued constant function on 
$\text{\rm Emb}_F(K,\overline F)$ or on $W_F$. 

The exact sequence
$$
 1\to\Bbb C^{\times}\to S\sphat\to T\sphat\to 1
$$
induces the exact sequence 
$$
 H_{\text{\rm conti}}^1(W_F,\Bbb C^{\times})\to 
 H_{\text{\rm conti}}^1(W_F,S\sphat)\to 
 H_{\text{\rm conti}}^1(W_F,T\sphat)\to 
 H_{\text{\rm conti}}^2(W_F,\Bbb C^{\times}).
$$
Since we have $H_{\text{\rm conti}}^2(W_F,\Bbb C^{\times})=\{1\}$ by 
\cite{Karpuk2013}, the canonical surjection $S\sphat\to T\sphat$
induces a canonical surjection
\begin{equation}
 H_{\text{\rm conti}}^1(W_F,S\sphat)\to 
 H_{\text{\rm conti}}^1(W_F,T\sphat)
\label{eq:canonical-surjection-of-1-cohomology-of-tori}.
\end{equation}

Then we have

\begin{prop}
\label{prop:local-langlands-correspondence-of-elliptic-tori}
There exists an isomorphism
\begin{equation}
 H_{\text{\rm conti}}^1(W_F,T\sphat)\,\tilde{\to}\,
 \text{\rm Hom}_{\text{\rm conti}}(U_{K/F},\Bbb C^{\times})
\label{eq:local-langlands-correspondence-of-elliptic-torus-sl(n)}
\end{equation}
such that the following diagram is commutative:
\begin{equation*}
\begin{tikzcd}
 H_{\text{\rm conti}}^1(W_F,S\sphat)
  \arrow{r}{\eqref{eq:continuous-local-langlands-correspondence-for-torus}}
  \arrow{d}[swap]{\eqref{eq:canonical-surjection-of-1-cohomology-of-tori}}
  &\text{\rm Hom}_{\text{\rm conti}}(K^{\times},\Bbb C^{\times})
    \arrow{d}{\eqref{eq:canonical-surjection-of-hom-group-of-elliptic-tori}}\\
 H_{\text{\rm conti}}^1(W_F,T\sphat)
  \arrow{r}[swap]{\eqref{eq:local-langlands-correspondence-of-elliptic-torus-sl(n)}}
  &\text{\rm Hom}_{\text{\rm conti}}(U_{K/F},\Bbb C^{\times})
\end{tikzcd}
\end{equation*}
\end{prop}
\begin{proof}
See \cite{Yu2009} for the arguments with general tori. A direct proof
for our specific setting is as follows. 

It is enough to show that 
an $\alpha\in Z^1_{\text{\rm conti}}(W_F,S\sphat)$ mapped to the
trivial element in $H^1(W_F,T\sphat)$ by 
\eqref{eq:canonical-surjection-of-1-cohomology-of-tori} if and only if 
$\alpha$ is mapped to the trivial element of 
$\text{\rm Hom}_{\text{\rm conti}}(U_{K/F},\Bbb C^{\times})$ by the
composite of 
\eqref{eq:continuous-local-langlands-correspondence-for-torus} and 
\eqref{eq:canonical-surjection-of-hom-group-of-elliptic-tori}. 
Let 
$c\in\text{\rm Hom}_{\text{\rm conti}}(K^{\times},\Bbb C^{\times})$ be
the image of $\alpha$ by 
\eqref{eq:continuous-local-langlands-correspondence-for-torus}, that
is, $c(x)=\alpha(\rho)(\text{\bf 1}_K)$ with 
$\rho\equiv\delta_K(x)\npmod{\overline{[W_K,W_K]}}$. By the
commutative diagram of the local class field theory
$$
\begin{tikzcd}
 W_K\arrow{r}
    \arrow{d}[swap]{\text{\rm inclusion}}
 &W_K/\overline{[W_K,W_K]}\arrow{r}{\sim}[swap]{\delta_K^{-1}}
  &K^{\times}\arrow{d}{N_{K/F}}\\
 W_F\arrow{r}
 &W_F/\overline{[W_F,W_F]}\arrow{r}{\sim}[swap]{\delta_F^{-1}}
  &F^{\times}
\end{tikzcd}
$$
$c$ is trivial on $U_{K/F}$ if and only if $c$ is extended to 
$\widetilde c\in
 \text{\rm Hom}_{\text{\rm conti}}(F^{\times},\Bbb C~{\times})$. If
so, put 
$s(\sigma)
 =\alpha(\sigma)(\text{\bf 1}_K)\cdot\widetilde c(\sigma)^{-1}$
for $\sigma\in W_F$. Then the cocycle condition of $\alpha$ implies
that $s$ is left $W_K$-invariant, and hence $s\in S\sphat$, and that 
$s(\sigma)\equiv s^{\sigma-1}\npmod{\Bbb C^{\times}}$ for all
$\sigma\in W_F$. Conversely if 
$\alpha(\sigma)\equiv s^{\sigma-1}\npmod{\Bbb C^{\times}}$ 
($\sigma\in W_F$) for some $s\in S\sphat$, then 
$\widetilde c=\alpha(\sigma)(\text{\bf 1}_K)$ ($\sigma\in W_F$) gives
an extension of $c$.
\end{proof}

Put $^L\!T=W_F\ltimes T\sphat$. Then a cohomology class 
$[\alpha]\in H^1_{\text{\rm conti}}(W_F,T\sphat)$ defines a continuous
group homomorphism 
\begin{equation}
 \widetilde\alpha:W_F\to\,^LT
 \quad
 (\sigma\mapsto(\sigma,\alpha(\sigma)))
\label{eq:langlands-parameter-for-torus}
\end{equation}
and $[\alpha]\mapsto\widetilde\alpha$ induces a well-defined bijection
$$
 H^1_{\text{\rm conti}}(W_F,T\sphat)\,\tilde{\to}\,
 \text{\rm Hom}_{\text{\rm conti}}^{\ast}(W_F,^LT)
 /\text{\rm ``$T\sphat$-conjugate"}
$$
where $\text{\rm Hom}_{\text{\rm conti}}^{\ast}(W_F,^LT)$ denotes the
set of the continuous group homomorphisms $\psi$ of $W_F$ to $^LT$
such that 
$W_F\xrightarrow{\psi}\,^LT\xrightarrow{\text{\rm proj.}}W_F$ is the
identity map.

\subsection{$\chi$-datum}
\label{subsec:chi-datum}
In this subsection, let us assume that $K/F$ is a Galois
extension and put $\Gamma=\text{\rm Gal}(K/F)$. For a 
$\gamma\in\Gamma$ of order two (if any), let us denote by $K_{\gamma}$
the intermediate subfield of $K/F$ such that 
$\text{\rm Gal}(K/K_{\gamma})=\langle\gamma\rangle$. 

Let $\Bbb S\sphat$ be the maximal torus of $GL_n(\Bbb C)$ consisting
of the non-singular diagonal matrices. Then 
$\Bbb T\sphat=\Bbb S\sphat/\Bbb C^{\times}$ is the maximal torus of 
the complex projective general linear group
$PGL_n(\Bbb C)=GL_n(\Bbb C)/\Bbb C^{\times}$. 

We have an isomorphism $T\sphat\,\tilde{\to}\,\Bbb T\sphat$ given by
$$
 s\mapsto\text{\rm diag}(s(\gamma_1),\cdots,s(\gamma_n))
$$
where 
$\text{\rm Emb}_F(K,\overline F)=\{\gamma_i\}_{1\leq i\leq n}$. 
The action
of $W_F$ on $T\sphat$ induces the action on $\Bbb T\sphat$ which
factors through $\Gamma$. 

The Weyl group 
$W(\Bbb S\sphat)
 =N_{GL_n(\Bbb C)}(\Bbb S\sphat)/\Bbb S\sphat$ on $\Bbb S\sphat$ is
 identified with a subgroup of the permutation group $S_n$. 
Then any $w\in W(\Bbb S\sphat)$ is represented by 
the permutation matrix $[w]\in GL_n(\Bbb C)$ corresponding to 
$w\in W(\Bbb S\sphat)\subset S_n$. Put 
$\widetilde w=[w]\npmod{\Bbb C^{\times}}
 \in N_{PGL_n(\Bbb C)}(\Bbb T\sphat)$. 

For any $\gamma\in\text{\rm Emb}_F(K,\overline F)=\Gamma$, 
let us denote by $a_{\gamma}$ an element of $X(S\sphat)$ such that 
$a_{\gamma}(s)=s(\gamma)$ for all $s\in T\sphat$. Put 
$a_i=a_{\gamma_i}\in X(\Bbb S\sphat)$ ($1\leq i\leq n$). Then 
$$
 \Phi(S\sphat)=\Phi(\Bbb S\sphat)
 =\left\{a_i\cdot a_j^{-1}\mid 1\leq i,j\neq n, i\neq j
        \right\}
$$
is the set of the
roots of $GL_n(\Bbb C)$ with respect to $S\sphat=\Bbb S\sphat$
with the simple roots
$$
 \Delta=\{\alpha_i
 =a_i\cdot a_{i+1}^{-1}\mid 1\leq i<n\}.
$$
Let $\{X_{\alpha},X_{-\alpha},H_{\alpha}\}$ be the standard triple
associate with a simple root $\alpha\in\Delta$. Then 
$s_{\alpha}\in W(\Bbb S\sphat)$ is represented by
$$
 n(s_{\alpha})
 =\exp(X_{\alpha})\cdot\exp(-X_{-\alpha})\cdot\exp(X_{\alpha})
 \in N_{GL_n(\Bbb C)}(\Bbb S\sphat)
$$
and $W(\Bbb S\sphat)$ is generated by
$\mathcal{S}=\{s_{\alpha}\}_{\alpha\in\Delta}$. 
For any $w\in W(\Bbb T\sphat)$, let $w=s_1s_2\cdots s_r$ 
($s_i\in \mathcal{S}$) be a reduced presentation and put
$$
 n(w)=n(s_1)n(s_2)\cdots n(s_r)
 \in N_{GL_n(\Bbb C)}(\Bbb T\sphat).
$$
Then 
$$
 r(w)=[w]^{-1}n(w)
     =\prod_{\stackrel{\scriptstyle \alpha<0,}
                      {\alpha^{w^{-1}}>0}}\varepsilon(\alpha)
     \in\Bbb T\sphat
$$
where 
$\prod_{\alpha<0,\alpha^{w^{-1}}>0}$ is the product over negative
roots $\alpha\in\Phi(\Bbb T\sphat)$ such that $\alpha^{w^{-1}}$ is
positive, and 
$\varepsilon(\alpha)\in\Bbb S\sphat$ with 
$\alpha=a_ia_j^{-1}\in\Phi(\Bbb T\sphat)$ is the diagonal matrix whose
diagonal elements are $1$ except $j$-the one is $-1$. 

The action of
$\sigma\in W_F$ on $X(S\sphat)$ induced from the action on $S\sphat$ is such
that $a_{\gamma}^{\sigma}=a_{\gamma\sigma}$ for all 
$\gamma\in\text{\rm Emb}_F(K,\overline F)$, and it determines an
element $w(\sigma)\in W(\Bbb S\sphat)$. Then
\cite{LanglandsShelstad1987} shows that the $2$-cocycle 
$t\in Z^2(W_F,\Bbb S\sphat)$ defined by
$$
 t(\sigma,\sigma^{\prime})
 =n(w(\sigma\sigma^{\prime}))^{-1}n(w(\sigma))\cdot
  n(w(\sigma^{\prime}))
 \quad
 (\sigma, \sigma^{\prime}\in W_F)
$$
is split by $r_p:W_F\to\Bbb S\sphat$ defined by $\chi$-data as follows.

For any $\lambda\in\Phi(S\sphat)$, put
$$
 \Gamma_{\lambda}
  =\{\sigma\in\Gamma\mid\lambda^{\sigma}=\lambda\},
 \quad
 \Gamma_{\pm\lambda}
  =\{\sigma\in\Gamma\mid\lambda^{\sigma}=\pm\lambda\}
$$
and put $F_{\lambda}=L^{\Gamma_{\lambda}}$, 
$F_{\pm\lambda}=L^{\Gamma_{\pm\lambda}}$. Then 
$(F_{\lambda}:F_{\pm\lambda})=\text{\rm $1$ or $2$}$ , 
and $\lambda$ is called symmetric if $(F_{\lambda}:F_{\pm\lambda})=2$.

The Galois group $\Gamma$ acts on $\Phi(S\sphat)$ and 
$$
 \Phi(S\sphat)/\Gamma
 =\{a_{\text{\bf 1}_K}a_{\gamma}^{-1}
                        \mid 1\neq\gamma\in\Gamma\}.
$$
If $\lambda=a_{\text{\bf 1}_K}a_{\gamma}^{-1}$, then $\lambda$ is
symmetric if and only if $\gamma^2=1$. In this case $F_{\lambda}=K$
and $F_{\pm\lambda}=K_{\gamma}$, and choose a continuous character 
$\chi_{\gamma}:F_{\lambda}^{\times}\to\Bbb C^{\times}$ such that 
$\chi_{\lambda}|_{F_{\pm\lambda}^{\times}}
 :K_{\gamma}^{\times}\to\{\pm 1\}$ is the character of the quadratic
extension $K/K_{\gamma}$. 

These characters are parts of a system of $\chi$-data
$\chi_{\lambda}:F_{\lambda}\to\Bbb C^{\times}$ 
($\lambda\in\Phi(\Bbb T\sphat)$) 
such that
\begin{enumerate}
\item $\chi_{-\lambda}=\chi_{\lambda}^{-1}$ and 
      $\chi_{\lambda^{\sigma}}=\chi_{\lambda}(x^{\sigma^{-1}})$ for
      all $\sigma\in\Gamma$, and
\item $\chi_{\lambda}=1$ if $\lambda$ is not symmetric.
\end{enumerate}
With this $\chi$-data and the gauge 
$$
 p:\Phi(\Bbb S\sphat)\to\{\pm 1\}\;\text{\rm s.t.}\;
 p(\lambda)=\begin{cases}
             1&:\lambda>0,\\
            -1&:\lambda<0,
            \end{cases}
$$
the mechanism of \cite{LanglandsShelstad1987} gives 
a $r_p:W_F\to\Bbb S\sphat$ such that
$$
 t(\sigma,\sigma^{\prime})
 =r_p(\sigma)^{\sigma^{\prime}}r_p(\sigma\sigma^{\prime})^{-1}
  r_p(\sigma^{\prime})
 \;\;
 \text{\rm for all $\sigma, \sigma^{\prime}\in W_F$}
$$
and
$$
 r_p(\sigma)
 =\prod_{\stackrel{\scriptstyle 1\neq\gamma\in\Gamma}
                  {\gamma^2=1}}
  \prod_{0<\lambda
         \in\{a_{\text{\bf 1}_K}\cdot a_{\gamma}^{-1}\}_{\Gamma}}
  \chi_{\lambda}(x)^{\check\lambda}
$$
if $\dot \sigma=(1,x)\in W_{K/F}=\Gamma\ltimes_{\alpha_{K/F}}K^{\times}$,
where $\{\alpha\}_{\Gamma}$ is the $\Gamma$-orbit of
$\alpha\in\Phi(\Bbb T\sphat)$ and $\check\lambda$ is the co-root of 
$\lambda$. Then we have a group homomorphism 
\begin{equation}
 ^LS=W_F\ltimes\Bbb S\sphat\to GL_n(\Bbb C)
 \quad
 ((\sigma,s)\mapsto n(w(\sigma))r_p(\sigma)^{-1}s).
\label{eq:langlands-shelstad-homomorphism-of-torus-to-dual-group}
\end{equation}
If we put $r(\sigma)=r(w(\sigma))$ for $\sigma\in W_F$, we have
$$
 t(\sigma,\sigma^{\prime})
 =r(\sigma)^{\sigma^{\prime}}r(\sigma\sigma^{\prime})^{-1}
  r(\sigma^{\prime})
 \qquad
 (\sigma, \sigma^{\prime}\in W_F).
$$
Now 
$\chi_p(\sigma)=r(\sigma)\cdot r_p(\sigma)^{-1}$ ($\sigma\in W_F$)
 define an element of $Z^1(W_F,\Bbb S\sphat)$ and the group
 homomorphism
 \eqref{eq:langlands-shelstad-homomorphism-of-torus-to-dual-group} is
\begin{equation}
 ^LS=W_F\ltimes\Bbb S\sphat\to GL_n(\Bbb C)
 \quad
 ((\sigma,s)\mapsto[w(\sigma)]\chi_p(\sigma)\cdot s).
\label{eq:modified-langlands-shelstad-homomorphism-of-torus-to-dual-group}
\end{equation}
Because the actions are trivial on the center $\Bbb C^{\times}$ of
$GL_n(\Bbb C)$, the group homomorphism 
\eqref{eq:langlands-shelstad-homomorphism-of-torus-to-dual-group} or 
\eqref{eq:modified-langlands-shelstad-homomorphism-of-torus-to-dual-group}
can be lifted to the group homomorphism
\begin{equation}
 ^LT=W_F\ltimes\Bbb T\sphat\to PGL_n(\Bbb C)
 \quad
 ((\sigma,s)\mapsto \widetilde w(\sigma)
                    \widetilde\chi_p(\sigma)\cdot s)
\label{eq:langlands-shelstad-homomorphism-of-torus-to-pgl(n)}
\end{equation}
where 
$\widetilde\chi_p(\sigma)
 =\chi_p(\sigma)\npmod{\Bbb C^{\times}}\in PGL_n(\Bbb C)$. 
Let $c\in \text{\rm Hom}_{\text{\rm conti}}(U_{K/F},\Bbb C^{\times})$
be the character corresponding to the cohomology class 
$[\widetilde\chi_p]
 \in H_{\text{\rm conti.}}^1(W_F,\Bbb T\sphat)$ by the local
Langlands correspondence of torus 
\eqref{eq:local-langlands-correspondence-of-elliptic-torus-sl(n)}. 
If we put 
\begin{equation}
 \widetilde c(x)
 =\chi_p(1,x)(\text{\bf 1}_K)
 =\prod_{\stackrel{\scriptstyle 1\neq\gamma\in\Gamma}
                   {\gamma^2=1}}
    \chi_{a_{\text{\bf 1}_K}\cdot a_{\gamma}^{-1}}(x)^{-1}
 \qquad
 (x\in K^{\times}),
\label{eq:explicit-formula-of-extension-of-c}
\end{equation}
then $c=\widetilde c|_{U_{K/F}}$. 
The following proposition will be used in the next two sections.

\begin{prop}\rm\label{prop:c-is-trivial-mod-p-square}
We can choose the $\chi$-data 
$\{\chi_{\lambda}\}_{\lambda\in\Phi(\Bbb T\sphat)}$ so that 
$c(x)=1$ for all $x\in U_{K/F}\cap(1+\frak{p}_K^2)$. 
\end{prop}
\begin{proof}
Take a $\gamma\in\Gamma$ of order two. Clearly 
$\left(1+\frak{p}_{K_{\gamma}},K/K_{\gamma}\right)=1$.
If $K/K_{\gamma}$ is ramified, then 
$(1+\frak{p}_K^2)\cap K_{\gamma}=1+\frak{p}_{K_{\gamma}}$, and if
$K/K_{\gamma}$ is unramified, then 
$(1+\frak{p}_K)\cap K_{\gamma}=1+\frak{p}_{K_{\gamma}}$. 
So we can assume that $\chi_{a_{\text{\bf 1}_K}a_{\gamma}^{-1}}$ is
trivial on $1+\frak{p}_K^2$. Then 
$c(x)=1$ for all $x\in U_{K/F}\cap(1+\frak{p}_K^2)$.
\end{proof}

\subsection{Explicit value of $c((-1)^{n-1})$}
\label{subsec:explicit-value-of-c((-1)(n-1))}
From now on, we will assume that $K/F$ is a tamely ramified
Galois extension and put $\Gamma=\text{\rm Gal}(K/F)$. 

We will prove the following proposition:

\begin{prop}\label{prop:explicit-value-of-c((-1)(n-1))}
According to the parity of $n=(K:F)$, we have
$$
 c((-1)^{n-1})=\begin{cases}
                1&:\text{\rm $n$=odd},\\
                (-1)^{\frac{q-1}2f}&:\text{\rm $n$=even}.
               \end{cases}
$$
\end{prop}

In order to prove the proposition, we need the structure of the set of
the order two elements of $\Gamma=\text{\rm Gal}(K/F)$. Put
\begin{equation}
 \text{\rm Gal}(K/F)=\langle\delta,\rho\rangle
\label{eq:generator-of-tame-galois-group}
\end{equation}
where $\text{\rm Gal}(K/K_0)=\langle\delta\rangle$ with the maximal
unramified subextension $K_0/F$ of $K/F$ and 
$\rho|_{K_0}\in\text{\rm Gal}(K_0/F)$ is the inverse of the Frobenius
automorphism. There exists a prime element $\varpi_K$ of $K$ such that
$\varpi_K^e\in K_0$. Then 
$\sigma\mapsto\varpi_K^{1-\sigma}\npmod{\frak{p}_K}$ is an
injective group homomorphism of $\text{\rm Gal}(K/K_0)$ into $\Bbb
K^{\times}$, and hence $e|q^f-1$. Put 
$\rho^f=\delta^m$ with $0\leq m<e$.
We have a relation $\rho^{-1}\delta\rho=\delta^q$ 
due to Iwasawa \cite{Iwasawa1955} and hence
$$
 \delta^m=\rho^{-1}\delta^m\rho=\delta^{qm}
$$
that is $m(q-1)\equiv 0\npmod{e}$. So we have
\begin{equation}
 \rho^{f(q-1)}=1
\label{eq:order-of-rho-divide-f(q-1)}
\end{equation}
Since $f$ divides $\text{\rm ord}(\rho)$, we have
$$
 \text{\rm ord}(\rho)=f\cdot\frac e{\text{\rm GCD}\{e,m\}}.
$$
The structure of the elements of order two in $\text{\rm Gal}(K/F)$
plays an important role in our arguments, and we have

\begin{prop}\label{prop:order-two-element-in-tamely-ramified-galois-group}
Assume that $|\text{\rm Gal}(K/F)|$ is even. Then
$$H=\{\gamma\in\text{\rm Gal}(K/F)\mid\gamma^2=1\}
  \subset Z(\text{\rm Gal}(K/F))
$$ 
and
$$
 H=\begin{cases}
    \{1,\delta^{\frac e2}\}&:f=\text{\rm odd}\;\text{\rm or}\;
                             \left\{\begin{array}{l}
                                     e=\text{\rm even},\\
                                     m=\text{\rm odd}
                                    \end{array}\right.\\
    \{1,\rho^{\frac f2}\delta^{-\frac m2}\}
                           &:e=\text{\rm odd}\,,m=\text{\rm even}\\
    \{1,\rho^{\frac f2}\delta^{\frac{e-m}2}\}
                           &:e=\text{\rm odd}\,,m=\text{\rm odd}\\
    \{1,\delta^{\frac e2},\rho^{\frac f2}\delta^{-\frac m2},
        \rho^{\frac f2}\delta^{\frac{e-m}2}\}
               &:f=\text{\rm even}\,,e=\text{\rm even}\,,
                 m=\text{\rm even}.
  \end{cases}
$$
For $\gamma\in\text{\rm Gal}(K/F)$ of order two, the quadratic
extension $K/K_{\gamma}$ is ramified  if and only if 
$\gamma\in\text{\rm Gal}(K/K_0)$.
\end{prop}
\begin{proof}
Take a $1\neq\gamma\in\text{\rm Gal}(K/F)$ such that $\gamma^2=1$. 

If $\gamma\in\text{\rm Gal}(K/K_0)$, then $e$ is even and
$\gamma=\delta^{\frac e2}$ 
is the unique element of order $2$ of the normal subgroup 
$\text{\rm Gal}(K/K_0)$. So $\gamma\in Z(\text{\rm Gal}(K/F))$. In
this case $K_0\subset K_{\gamma}$ and $K/K_{\gamma}$ is ramified
extension. 

Assume that $\gamma\not\in\text{\rm Gal}(K/K_0)$. Then 
$\gamma|_{K_0}\in\text{\rm Gal}(K_0/F)$ is of order two 
(hence $f=2f^{\prime}$ is even), 
and $\gamma=\rho^{f^{\prime}}\delta^a$ with $0\leq a<e$. 
Then $K/K_{\gamma}$ is unramified
  extension, because if it was not the case we have $f(K_{\gamma}/F)=f(K/F)$
  and hence $K_0\subset K_{\gamma}$ which means 
$$
 \gamma\in\text{\rm Gal}(K/K_{\gamma})\subset\text{\rm Gal}(K/K_0),
$$
contradicting to the assumption 
$\gamma\not\in\text{\rm Gal}(K/K_0)$.
Then $f(K_{\gamma}/F)=f^{\prime}$ and $e(K_{\gamma}/F)=e$, and hence 
$e|q^{f^{\prime}}-1$. So we have
$$
 1=\gamma^2=\rho^f\rho^{-f^{\prime}}\delta^a\rho^{f^{\prime}}\delta^a
  =\delta^{m+aq^{f^{\prime}}+a}
  =\delta^{2a+m},
$$
hence $2a\equiv -m\npmod{e}$. Then 
$a\equiv -\frac m2\;\text{\rm or}\;\frac{e-m}2\npmod{e}$ 
if $e$ is even (hence $m$ is even), and
$$
 a\npmod{e}=\begin{cases}
             -\frac m2&:\text{\rm if $m$ is even},\\
             \frac{e-m}2&:\text{\rm if $m$ is odd}
            \end{cases}
$$
if $e$ is odd. We have $e|q^{f^{\prime}}-1$ hence
$$
 \delta\gamma
 =\rho^{f^{\prime}}\delta^{q^{f^{\prime}}+a}
 =\rho^{f^{\prime}}\delta^{1+a}
 =\gamma\delta.
$$
Now we have
\begin{equation}
 \rho^{f^{\prime}(q-1)}=1.
\label{eq:f-prime-(q-1)-kill-rho}
\end{equation}
In fact 
$\text{\rm Gal}(K_{\gamma}/F)
 =\langle\delta^{\prime},\rho^{\prime}\rangle$ with 
$\delta^{\prime}=\delta|_{K_{\gamma}}, \rho^{\prime}=\rho|_{K_{\gamma}}$. Then 
$(\rho^{\prime})^{f^{\prime}(q-1)}=1$, that is
$$
 \rho^{f^{\prime}(q-1)}
 \in\text{\rm Gal}(K/K_{\gamma})=\langle\gamma\rangle.
$$
If $\rho^{f^{\prime}(q-1)}\neq 1$, then 
$\rho^{f^{\prime}(q-1)}=\gamma=\rho^{f^{\prime}}\delta^a$, 
therefore 
$$
 \rho^{f^{\prime}q}=\rho^f\delta^a=\delta^{m+a}
 \in\text{\rm Gal}(K/K_0)
$$
and hence $f$ divides $f^{\prime}q$, contradicting to the assumption
that $q$ is odd. Now we have
$$
 \gamma\rho=\rho^{f^{\prime}+1}\delta^{qa}
 =\rho\gamma\cdot\delta^{a(q-1)}.
$$
For $a=-\frac m2$ or $a=\frac{e-m}2$, 
we have $a(q-1)\equiv 0\npmod{e}$ if and only if 
$$
 \frac{q-1}2\equiv 0\npmod{\frac e{\text{\rm GCD}\{e,m\}}}
$$
which is equivalent to 
$\rho^{f\cdot\frac{q-1}2}=\rho^{f^{\prime}(q-1)}=1$. 
Then \eqref{eq:f-prime-(q-1)-kill-rho} implies
$\gamma\rho=\rho\gamma$. Then we have $\gamma$ is an element of the
center of $\text{\rm Gal}(K/F)$. 
\end{proof}

\noindent
{\bf Proof of Proposition 
\ref{prop:explicit-value-of-c((-1)(n-1))}} 
Assume the $n=(K:F)$ is even. Then
$$
 c(-1)=\prod_{\stackrel{\scriptstyle 1\neq\gamma\in\Gamma}
                       {\gamma^2=1}}
         \left(-1,K/K_{\gamma}\right)
$$
by \eqref{eq:explicit-formula-of-extension-of-c}. Here 
$$
 \left(-1,K/K_{\gamma}\right)
 =\begin{cases}
   1&:\text{\rm $K/K_{\gamma}$ is unramified},\\
  (-1)^{\frac{q^f-1}2}
    &:\text{\rm $K/K_{\gamma}$ is ramified}
  \end{cases}
$$
for the quadratic extension $K/K_{\gamma}$ with $1\neq\gamma\in\Gamma$
such that $\gamma^2=1$. The extension $K/K_{\gamma}$ is ramified if
and only if $\gamma\in\text{\rm Gal}(K/K_0)=\langle\delta\rangle$, the
list of Proposition 
\ref{prop:order-two-element-in-tamely-ramified-galois-group} shows
that
$$
 c(-1)=\begin{cases}
        1&:\text{\rm $e$=odd},\\
       (-1)^{\frac{q^f-1}2}&:\text{\rm $e$=even}.
       \end{cases}
$$
Note that
$$
 \frac{q^f-1}2
 =\frac{q-1}2(1+q+q^2+\cdots+q^{f-1})
 \equiv\frac{q-1}2f\npmod{2}.
$$
This complete the proof.

\subsection{$L$-parameters associated with characters of tame elliptic
            tori}
\label{subsec:l-parameter-associated-with-character-of-tame-elliptic-tori}
By local Langlands correspondence of tori described in Proposition 
\ref{prop:local-langlands-correspondence-of-elliptic-tori}, 
the continuous character $\theta$ of $U_{K/F}$ which parametrizes
the irreducible representation $\delta_{\beta,\theta}$ of 
$SL_n(O_F)$ determines, by choosing an extension of $\theta$ to
$K^{\times}$, the cohomology class 
$[\alpha_1]\in H^1_{\text{\rm conti}}(W_F,S\sphat)$ which is mapped
onto 
$[\alpha]\in H^1_{\text{\rm conti}}(W_F,T\sphat)$ 
by \eqref{eq:canonical-surjection-of-1-cohomology-of-tori}, which is
independent of the choice of the extension to $K^{\times}$ of
$\theta$. Then we have a group homomorphisms
\begin{equation}
 \varphi_1:W_F\xrightarrow{\widetilde\alpha_1}
           {^LS}\xrightarrow{
   \eqref{eq:langlands-shelstad-homomorphism-of-torus-to-dual-group}}
          GL_n(\Bbb C)
\label{eq:langlands-parameter-of-kaletha-gl(n)}
\end{equation}
and
\begin{equation}
 \varphi:W_F\xrightarrow{\widetilde\alpha}
         {^LT}\xrightarrow{
   \eqref{eq:langlands-shelstad-homomorphism-of-torus-to-pgl(n)}}
         PGL_n(\Bbb C).
\label{eq:langlands-parameter-of-kaletha-pgl(n)}
\end{equation}
The construction of $\varphi$ shows that 
$\varphi(\sigma)=\varphi_1(\sigma)\npmod{\Bbb C^{\times}}$ 
($\sigma\in W_F$). The definition of 
\eqref{eq:modified-langlands-shelstad-homomorphism-of-torus-to-dual-group}
shows that 
\begin{align*}
 \text{\rm tr}\,\varphi_1(\sigma)
 &=\sum_{\gamma\in\text{\rm Emb}_F(K,\overline F),
          \gamma\sigma=\gamma}
    \chi_p(\sigma)(\gamma)\cdot\alpha(\sigma)(\gamma)\\
 &=\sum_{\dot\gamma\in W_K\backslash W_F, 
          \gamma\sigma\gamma^{-1}\in W_K}
    \chi_p(\sigma)(\gamma)\cdot\alpha(\sigma(\gamma)\\
 &=\sum_{\dot\gamma\in W_K\backslash W_F, 
          \gamma\sigma\gamma^{-1}\in W_K}
    \psi_c\cdot\psi_{\theta}(\gamma\sigma\gamma^{-1})
\end{align*}
for $\sigma\in W_F$.
Here $\psi_c$ and $\psi_{\theta}$ are respectively the elements of 
$\text{\rm Hom}_{\text{\rm conti}}(W_K,\Bbb C)$ corresponding to 
$\widetilde c$ defined by 
\eqref{eq:explicit-formula-of-extension-of-c} and an extension 
$\widetilde\theta
 \in\text{\rm Hom}_{\text{\rm conti}}(K^{\times},\Bbb C^{\times})$ of
 $\theta$ by the isomorphism of the local class field theory
$$
 W_K\to W_K/\overline{[W_K,W_K]}
    \xrightarrow[\delta_K^{-1}]{\sim}K^{\times}.
$$
This shows that 
$\varphi_1$ is the
induced representation of $W_F$ from the character
$\psi_c\cdot\psi_{\theta}$ of $W_K$. So $\varphi_1$ factors through the
canonical surjection 
$$
 W_F\to W_{K/F}=W_F/\overline{[W_K,W_K]}
$$ 
and, if we put $\vartheta=c\cdot\theta$ and 
$\widetilde\vartheta(x)=\widetilde c\cdot\widetilde\theta$, the
extension of $\vartheta$, then we have
\begin{equation}
 \text{\rm tr}\,\varphi_1(\sigma,x)
 =\begin{cases}
   0&:\sigma\neq 1,\\
   \sum_{\gamma\in\text{\rm Gal}(K/F)}
     \widetilde{\vartheta}(x^{\gamma})
    &:\sigma=1
  \end{cases}
\label{eq:chracter-of-o(2n)-part-of-l-parameter}
\end{equation}
for 
$(\sigma,x)\in W_{K/F}
 =\text{\rm Gal}(K/F)\ltimes_{\alpha_{K/F}}K^{\times}$ with the
 fundamental class 
$[\alpha_{K/F}]\in H^2(\text{\rm Gal}(K/F),K^{\times})$.

The representation space $V_{\vartheta}$ of the induced representation 
$\text{\rm Ind}_{K^{\times}}^{W_{K/F}}\widetilde\vartheta$ is the
complex vector space of the $\Bbb C$-valued function $v$ on 
$\text{\rm Gal}(K/F)$ with the action of $(\sigma,x)\in W_{K/F}$
\begin{equation}
 (x\cdot v)(\gamma)
 =\widetilde\vartheta(x^{\gamma})\cdot v(\gamma),
 \quad
 (\sigma\cdot v)(\gamma)
 =\widetilde\vartheta(\alpha_{K/F}(\sigma,\sigma^{-1}\gamma))
    \cdot v(\sigma^{-1}\gamma).
\label{eq:action-of-relative-weil-group-on-induced-rep}
\end{equation}
A $\Bbb C$-basis $\{v_{\rho}\}_{\rho\in\text{\rm Gal}(K/F)}$ 
of $V_{\vartheta}$ is defined by
$$
 v_{\rho}(\gamma)=\begin{cases}
                     1&:\gamma=\rho,\\
                     0&:\gamma\neq\rho.
                    \end{cases}
$$
Then 
$$
 x\cdot v_{\rho}=\widetilde\vartheta(x^{\rho})\cdot v_{\rho},
 \quad
 \sigma\cdot v_{\rho}
 =\widetilde\vartheta(\alpha_{K/F}(\sigma,\rho))\cdot 
   v_{\sigma\rho}
$$
for $(\sigma,x)\in W_{K/F}$. The following proposition will be used to
analyze $\text{\rm Ind}_{K^{\times}}^{W_{K/F}}\widetilde\vartheta$ in
detail. 

\begin{prop}
\label{prop:conductor-and-conjugate-triviality-of-tilde-vartheta}
Assume $l\geq 2$, then for an integer $k\geq 2$
$$
 \left\{\sigma\in\text{\rm Gal}(K/F)\biggm|
   \begin{array}{l}
    \widetilde\vartheta(x^{\sigma})=\widetilde\vartheta(x)\\
     \text{\rm for $\forall x\in 1+\frak{p}_K^k$}
   \end{array}\right\}
 =\begin{cases}
   \text{\rm Gal}(K/F)&:k>e(r-1),\\
   \text{\rm Gal}(K/K_0)&:k=e(r-1),\\
   \{1\}&:k<e(r-1).
  \end{cases}
$$
\end{prop}
\begin{proof}
Note that $\vartheta(x)=\theta(x)$ for all 
$x\in U_{K/F}\cap(1+\frak{p}_K^2)$ 
(by Proposition \ref{prop:c-is-trivial-mod-p-square}) and 
$\theta(x)=1$ for all $x\in U_{K/F}\cap(1+\frak{p}_K^{er})$.
Take an integer $k$ such that $0\leq k\leq el^{\prime}$, and hence 
$2\leq el\leq er-k$. Then, for any $x\in O_K$, we have
$$
 (1+\varpi_F^r\varpi_K^{-k}x)^{1-\tau}
 \equiv
 1+\varpi_F^r(\varpi_K^{-k}x-\varpi_K^{-k\tau}x^{\tau})
  \npmod{\frak{p}_K^{er}}
$$
since $2(er-k)\geq er$. Hence, for 
$\alpha=1+\varpi_F\varpi_K^{-k}x\in 1+\frak{p}_K^{er-k}$ ($x\in O_K$), 
we have
\begin{equation}
 \widetilde\vartheta(\alpha)
  =\psi\left(T_{K/F}\left(
        (\varpi_K^{-k}x-\varpi_K^{-k\tau}x^{\tau})\beta\right)\right)
  =\psi\left(2T_{K/F}(\varpi_K^{-k}x\beta)\right). 
\label{eq:explicit-value-of-tilde-vartheta-alpha}
\end{equation}
Because $K/F$ is tamely ramified, we have
\begin{align}
 V_t(K/F)
 &=\{\sigma\in\text{\rm Gal}(K/F)\mid
             \text{\rm ord}_K(x^{\sigma}-x)\geq t+1\;
              \forall x\in O_K\} \nonumber\\
 &=\begin{cases}
    \text{\rm Gal}(K/F)&:t<0,\\
    \text{\rm Gal}(K/K_0)&:0\leq t<1,\\
    \{1\}&:1\leq t.
   \end{cases}
\label{eq:ramification-group-of-tamely-ramified-extension}
\end{align}
Take a $\sigma\in\text{\rm Gal}(K/F)$. Then, by 
\eqref{eq:explicit-value-of-tilde-vartheta-alpha}, we have
$$
 \widetilde\vartheta(\alpha^{\sigma})
 =\psi\left(2T_{K/F}(\varpi_K^{-k\sigma}x^{\sigma}\beta)\right)
  =\psi\left(2T_{K/F}(\varpi_K^{-k}x\beta^{\sigma})\right).
$$
So the statement 
$\widetilde\vartheta(\alpha^{\sigma})=\widetilde\vartheta(\alpha)$ for
all $\alpha\in 1+\frak{p}_K^{er-k}$ is equivalent to the statement 
$\varpi_K^{-k}(\beta^{\sigma}-\beta)
 \in\mathcal{D}(K/F)^{-1}=\frak{p}_K^{1-e}$, or to the statement 
$$
 \text{\rm ord}_K(x^{\sigma}-x)\geq k-e+1
 \;\;\text{\rm for all $x\in O_K$}
$$
since $O_K=O_F[\beta]$, which is equivalent to 
$\sigma\in V_{k-e}$. Then
\eqref{eq:ramification-group-of-tamely-ramified-extension} completes
the proof.
\end{proof}

\begin{prop}
\label{prop:induced-representation-of-vartheta-is-irreducible}
The induced representation 
$\text{\rm Ind}_{K^{\times}}^{W_{K/F}}\widetilde\vartheta$ is
irreducible.
\end{prop}
\begin{proof}
Take a $0\neq T\in\text{\rm End}_{W_{K/F}}(V_{\vartheta})$. Since
$$
 Tv_{\rho}=T(\rho\cdot v_1)=\rho\cdot Tv_1
$$
for all $\rho\in\text{\rm Gal}(K/F)$, we have $Tv_1\neq 0$. If 
$(Tv_1)(\gamma)\neq 0$ for a $\gamma\in\text{\rm Gal}(K/F)$, then we
have
\begin{align*}
 \widetilde\vartheta(x^{\gamma})\cdot(Tv_1)(\gamma)
 &=(x\cdot Tv_1)(\gamma)=T(x\cdot v_1)(\gamma)\\
 &=(T(\widetilde\vartheta(x)\cdot v_1))(\gamma)
  =\widetilde\vartheta(x)\cdot(Tv_1)(\gamma),
\end{align*}
and hence 
$\widetilde\vartheta(x^{\gamma})=\widetilde\vartheta(x)$ 
for all $x\in K^{\times}$. Then $\gamma=1$ by Proposition 
\ref{prop:conductor-and-conjugate-triviality-of-tilde-vartheta}. This
means $Tv_1=c\cdot v_1$ with a $c\in\Bbb C^{\times}$. Then
$$
 Tv_{\rho}=\rho\cdot(Tv_1)=c\cdot v_{\rho}
$$
for all $\rho\in\text{\rm Gal}(K/F)$, and hence $T$ is a homothety.
\end{proof}

\begin{rem}\label{remark:condition-for-irreducibility-of-induced-rep}
The proof of Proposition
\ref{prop:induced-representation-of-vartheta-is-irreducible} shows
that the induced representation 
$\text{\rm Ind}_{K^{\times}}^{W_{K/F}}\widetilde\vartheta$ is
irreducible if $\widetilde\vartheta$ is a character of $K^{\times}$
such that 
$\widetilde\vartheta(x^{\sigma})=\widetilde\vartheta(x)$ for all 
$x\in K^{\times}$ with $\sigma\in\text{\rm Gal}(K/F)$ implies
$\sigma=1$. 
\end{rem}

%\subsection{Induced representations of Weil group}
%\label{subsec:induced-representation-of-weil-group}

\section{Formal degree conjecture}
\label{sec:formal-degree-conjecture}
In this section, we will assume that $K/F$ is a tamely ramified Galois
extension such that the degree $(K:F)=n$ is prime to $p$, and will
keep the notations of preceding sections. 

\subsection{$\gamma$-factor of adjoint representation}
\label{subsec:gamma-factor-of-adjoint-representation}
The admissible representation of the Weil-Deligne group 
$W_F\times SL_2(\Bbb C)$ to $PGL_n(\Bbb C)$
corresponding to the triple $(\varphi,PGL_n(\Bbb C),0)$ as presented
in the appendix \ref{subsec:weil-deligne-group} is 
\begin{equation}
 W_F\times SL_2(\Bbb C)
 \xrightarrow{\text{\rm projection}} W_F
 \xrightarrow{\varphi} PGL_n(\Bbb C)
\label{eq:candidate-of-langlands-arthur-parameter}
\end{equation}
which is denoted also by $\varphi$. We will also denote by $\varphi_1$
the representation
$$
 W_F\times SL_2(\Bbb C)
 \xrightarrow{\text{\rm projection}} W_F
 \xrightarrow{\varphi_1} GL_n(\Bbb C).
$$
Then, as a complex linear representation of $W_F\times SL_2(\Bbb C)$, 
the adjoint representation of $PGL_n(\Bbb C)$ on its Lie algebra
composed with the $\varphi$ is equivalent to the
adjoint representation of $GL_n(\Bbb C)$ on 
$$
 \widehat{\frak g}=\frak{sl}_n(\Bbb C)
 =\{X\in M_n(\Bbb C)\mid \text{\rm tr}(X)=0\}
$$
composed with the $\varphi_1$. 

The purpose of this subsection is to determine the $\gamma$-factor 
$\gamma(\varphi,\text{\rm Ad},\psi,d(x),s)$, as explained in the
appendix \ref{subsec:weil-deligne-group}. 

Let us use the notation of \eqref{eq:generator-of-tame-galois-group}
$$
 \text{\rm Gal}(K/F)=\langle\delta,\rho\rangle,
$$
that is, $\text{\rm Gal}(K/K_0)=\langle\delta\rangle$ with the maximal
unramified subextension $K_0/F$ of $K/F$ and 
$\rho|_{K_0}\in\text{\rm Gal}(K_0/F)$ is the inverse of the Frobenius
automorphism. By the canonical surjection
$$
 W_F\to W_F/\overline{[W_K,W_K]}=W_{K/F}
        =\text{\rm Gal}(K/F){\ltimes}_{\alpha_{K/F}}K^{\times}
        \subset\text{\rm Gal}(K^{\text{\rm ab}}/F),
$$
$I_F=\text{\rm Gal}(F^{\text{\rm alg}}/F^{\text{\rm ur}})\subset W_F$
     is mapped onto 
$$
 \text{\rm Gal}(K/K_0){\ltimes}_{\alpha_{K/F}}O_K^{\times}
 =\text{\rm Gal}(K^{\text{\rm ab}}/F^{\text{\rm ur}}).
$$

Put $v_{ij}=v_{\delta^{i-1}\rho^{j-1}}\in V_{\vartheta}$ 
($1\leq i\leq e,1\leq j\leq f$), and identify $GL_{\Bbb C}(V_{\vartheta})$
with $GL_n(\Bbb C)$ by the $\Bbb C$-basis 
$\{v_{ij}\}_{1\leq i\leq e, 1\leq j\leq f}$ of 
$V_{\vartheta}
 =\text{\rm Ind}_{K^{\times}}^{W_{K/F}}\widetilde\vartheta$. Then the
 action \eqref{eq:action-of-relative-weil-group-on-induced-rep} gives
\begin{equation}
 \varphi_1(\delta)=\begin{bmatrix}
                  J_1&   &      &       \\
                     &J_2&      &       \\
                     &   &\ddots&       \\
                     &   &      &J_f
                 \end{bmatrix}
 \in GL_n(\Bbb C)
\label{eq:explicit-action-of-theta-tau-0}
\end{equation}
with 
$$
 J_j=\begin{bmatrix}
          0  &   0  &\cdots&  0      &  1   \\
          1  &   0  &      &         &  0   \\
             &   1  &\ddots&         &\vdots\\
             &      &\ddots&  0      &  0   \\
             &      &      &  1      &  0
     \end{bmatrix}
     \begin{bmatrix}
      a_{1j}&      &      &         \\
            &a_{2j}&      &         \\
            &      &\ddots&         \\
            &      &      &a_{e,j}
     \end{bmatrix}
$$
($a_{ij}
  =\theta\left(\alpha_{K/F}(\delta,\delta^{i-1}\rho^{j-1})\right)$). 
Since 
the action of $O_K^{\times}$ on $V_{\vartheta}$ is diagonal, the space 
$\widehat{\frak g}^{I_F}$ of the 
$\text{\rm Ad}\circ\varphi(I_F)$-fixed vectors is  
$$
 \left\{\begin{bmatrix}
          a_11_e&      &      &      \\
                &a_21_e&      &      \\
                &      &\ddots&      \\
                &      &      &a_f1_e
         \end{bmatrix}\biggm| 
          \begin{array}{l}
           a_i\in\Bbb C,\\
           a_1+a_2+\cdots+a_f=0
          \end{array}\right\}.
$$
A $\Bbb C$-basis of it is given by
\begin{equation}
 X_1=\begin{bmatrix}
      P&   &      & \\
       &0_e&      & \\
       &   &\ddots& \\
       &   &      &0_e
     \end{bmatrix},\;
 X_2=\begin{bmatrix}
       0_e& &      & \\
          &P&      & \\
          & &\ddots& \\
          & &      &0_e
      \end{bmatrix},\cdots,
  X_{f-1}=\begin{bmatrix}
           0_e&      &   &  \\
              &\ddots&   &  \\
              &      &0_e&  \\
              &      &   &P
          \end{bmatrix}
\label{eq:baisi-of-ad-varphi-i-f-fixed-subspec-of-widehat-g}
\end{equation}
with $P=\begin{bmatrix}
         1_e&   \\
            &-1_e
        \end{bmatrix}$.
Since $\rho\delta\rho^{-1}=\delta^l$ with 
$0<l<e$ and $ql\equiv 1\npmod{e}$, we have
$$
 \rho\delta^{i-1}\rho^{j-1}
 =\begin{cases}
   \delta^{l(i-1)}\rho^j&:1\leq j<f,\\
   \delta^{l(i-1)+m}&:j=f.
  \end{cases}.
$$
Put 
$$
 l(i-1)\equiv i^{\prime}-1\npmod{e},
 \quad
 l(i-1)+m\equiv i^{\prime\prime}-1\npmod{e}
 \quad
 (1\leq i^{\prime}, i^{\prime\prime}\leq e)
$$ 
for $0\leq i<e$ and let 
$[l]_e, [l,m]_e\in GL_e(\Bbb Z)$ be the permutation matrix associated
respectively with the element 
$$
 \begin{pmatrix}
  1&1&2&\cdots&e\\
  1^{\prime}&1^{\prime}&2^{\prime}&\cdots&e^{\prime}
 \end{pmatrix},
 \quad
 \begin{pmatrix}
  1&1&2&\cdots&e\\
  1^{\prime\prime}&1^{\prime\prime}&2^{\prime\prime}
                                    &\cdots&e^{\prime\prime}
 \end{pmatrix}
$$
of the symmetric group of degree $e$. Then we have
$$
 \varphi_1(\rho)=\overline{
                   \begin{bmatrix}
                    0  & 0 &\cdots&   0   &I_f\\
                    I_1& 0 &      &       &   0   \\
                       &I_2&\ddots&       & \vdots\\
                       &   &\ddots&   0   &   0   \\
                       &   &      &I_{f-1}&   0
                   \end{bmatrix}}
$$
with 
$$
 I_j=P_j\cdot\begin{bmatrix}
              b_{1j}&      &      &         \\
                    &b_{2j}&      &         \\
                    &      &\ddots&         \\
                    &      &      &b_{e,j}
             \end{bmatrix},
 \quad
 P_j=\begin{cases}
      [l]_e&:1\leq j<f,\\
      [l,m]_e&:j=f
     \end{cases}
$$
($b_{ij}
  =\theta\left(\alpha_{K/F}(\rho,\delta^{i-1}\rho^{j-1})\right)$). So
 the representation matrix of 
$\text{\rm Ad}\circ\varphi(\widetilde{\text{\rm Fr}})$ on 
${\widehat{\frak{g}}}^{I_F}$ with respect
 to the basis 
\eqref{eq:baisi-of-ad-varphi-i-f-fixed-subspec-of-widehat-g} is
$$
  \begin{bmatrix}
     -1  &   1  &  0   &\cdots&  0   \\
     -1  &   0  &  1   &\cdots&  0   \\
   \vdots&\vdots&\ddots&\ddots&\vdots\\
     -1  &   0  &      &   0  &  1   \\
     -1  &   0  &\cdots&   0  &  0
  \end{bmatrix}.
$$
Hence we have
\begin{align*}
 L(\varphi,\text{\rm Ad},s)
 &=\det\left(
   1-q^{-s}\text{\rm Ad}\circ\varphi(\widetilde{\text{\rm Fr}})
    |_{\frak{g}^{\text{\rm Ad}\circ\varphi(\widetilde{\text{\rm Fr}})}}
      \right)^{-1}\\
 &=\left(1+q^{-s}+q^{-2s}+\cdots+q^{-(f-1)s}\right)^{-1}
\end{align*}
and 
\begin{equation}
 \frac{L(\varphi,\text{\rm Ad},1)}
      {L(\varphi,\text{\rm Ad},0)}
 =f\cdot\frac{1-q^{-1}}
            {1-q^{-f}}.
\label{eq:ratio-of-l-function-of-tamely-induced-rep-of-weil-group}
\end{equation}
Let us denote by $K^{(k)}=K_{\varpi_K,k}$ ($k=1,2,\cdots$) the
field of $\varpi_K^k$-th division points of Lubin-Tate theory over
$K$. Then we have an isomorphism 
$$
 \delta_K:1+\frak{p}_K^k\,\tilde{\to}\,
  \text{\rm Gal}(K^{\text{\rm ab}}/K^{(k)}K^{\text{\rm ur}}).
$$
Because the character 
$\widetilde\vartheta:K^{\times}\to\Bbb C^{\times}$ comes from
a character of 
$$
 G_{\beta}(O_F/\frak{p}^r)\subset\left(O_K/\frak{p}_K^{er}\right)^{\times},
$$
$\varphi$ is trivial on 
$\text{\rm Gal}(K^{\text{\rm ab}}/K^{(er)}K^{\text{\rm ur}})$. 
Note that 
$K^{(er)}K^{\text{\rm ur}}=K^{(er)}F^{\text{\rm ur}}$ 
is a finite extension of $F^{\text{\rm ur}}$. If us use the upper
numbering 
$$
 V^s=V_t(K^{(er)}F^{\text{\rm ur}}/F^{\text{\rm ur}})
$$
of the higher ramification group, where $t\mapsto s$
is the inverse of Hasse function whose graph is
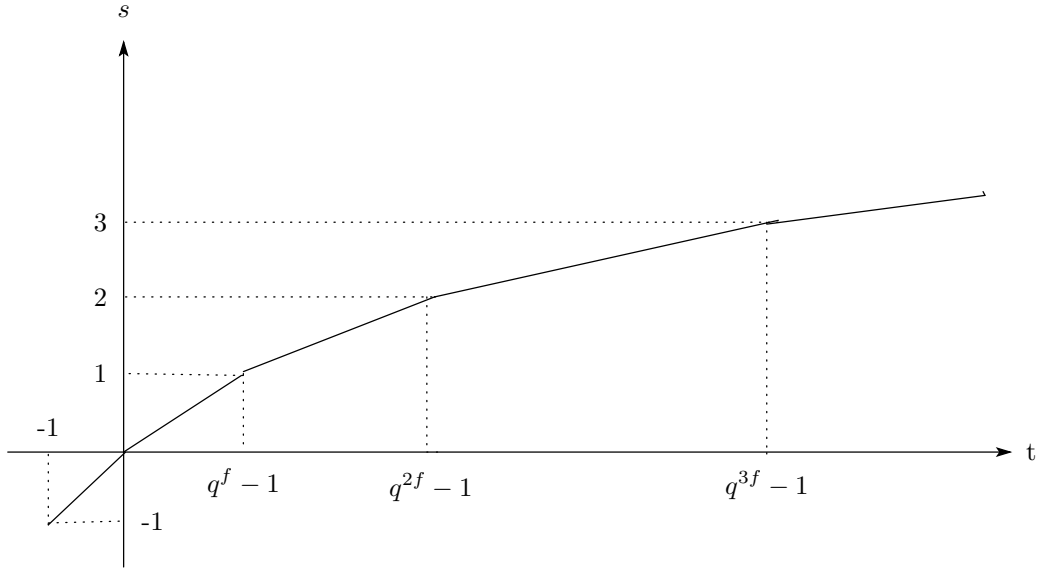
\begin{center}
%WinTpicVersion4.26
\unitlength 0.1in
\begin{picture}( 52.8000, 29.6500)( 14.0000,-42.1000)
% LINE 2 0 3 0 Black White
% 2 2000 3610 2610 3210
% 
\special{pn 8}%
\special{pa 2000 3610}%
\special{pa 2610 3210}%
\special{fp}%
% LINE 2 0 3 0 Black White
% 4 2620 3190 3610 2800 3610 2800 3610 2800
% 
\special{pn 8}%
\special{pa 2620 3190}%
\special{pa 3610 2800}%
\special{fp}%
\special{pa 3610 2800}%
\special{pa 3610 2800}%
\special{fp}%
% LINE 2 0 3 0 Black White
% 2 3610 2800 5390 2400
% 
\special{pn 8}%
\special{pa 3610 2800}%
\special{pa 5390 2400}%
\special{fp}%
% LINE 2 0 3 0 Black White
% 2 2010 3610 1610 3990
% 
\special{pn 8}%
\special{pa 2010 3610}%
\special{pa 1610 3990}%
\special{fp}%
% LINE 2 2 3 0 Black White
% 2 1610 3990 1610 3610
% 
\special{pn 8}%
\special{pa 1610 3990}%
\special{pa 1610 3610}%
\special{dt 0.045}%
% LINE 2 2 3 0 Black White
% 2 1610 3980 2010 3970
% 
\special{pn 8}%
\special{pa 1610 3980}%
\special{pa 2010 3970}%
\special{dt 0.045}%
% LINE 2 2 3 0 Black White
% 2 2620 3200 2620 3600
% 
\special{pn 8}%
\special{pa 2620 3200}%
\special{pa 2620 3600}%
\special{dt 0.045}%
% LINE 2 2 3 0 Black White
% 2 2620 3210 2010 3200
% 
\special{pn 8}%
\special{pa 2620 3210}%
\special{pa 2010 3200}%
\special{dt 0.045}%
% LINE 2 2 3 0 Black White
% 4 2010 2800 3620 2800 3620 3610 3620 3610
% 
\special{pn 8}%
\special{pa 2010 2800}%
\special{pa 3620 2800}%
\special{dt 0.045}%
\special{pa 3620 3610}%
\special{pa 3620 3610}%
\special{dt 0.045}%
% LINE 2 2 3 0 Black White
% 6 2010 2410 5400 2410 5330 2410 5330 3610 5330 3620 5330 3580
% 
\special{pn 8}%
\special{pa 2010 2410}%
\special{pa 5400 2410}%
\special{dt 0.045}%
\special{pa 5330 2410}%
\special{pa 5330 3610}%
\special{dt 0.045}%
\special{pa 5330 3620}%
\special{pa 5330 3580}%
\special{dt 0.045}%
% VECTOR 2 0 3 0 Black White
% 2 2000 4210 2000 1470
% 
\special{pn 8}%
\special{pa 2000 4210}%
\special{pa 2000 1470}%
\special{fp}%
\special{sh 1}%
\special{pa 2000 1470}%
\special{pa 1980 1538}%
\special{pa 2000 1524}%
\special{pa 2020 1538}%
\special{pa 2000 1470}%
\special{fp}%
% STR 2 0 3 0 Black White
% 4 1880 3100 1880 3200 5 0 0 0
% 1
\put(18.8000,-32.0000){\makebox(0,0){1}}%
% STR 2 0 3 0 Black White
% 4 1880 2700 1880 2800 5 0 0 0
% 2
\put(18.8000,-28.0000){\makebox(0,0){2}}%
% STR 2 0 3 0 Black White
% 4 1880 2310 1880 2410 5 0 0 0
% 3
\put(18.8000,-24.1000){\makebox(0,0){3}}%
% STR 2 0 3 0 Black White
% 4 2150 3870 2150 3970 5 0 0 0
% -1
\put(21.5000,-39.7000){\makebox(0,0){-1}}%
% STR 2 0 3 0 Black White
% 4 1610 3380 1610 3480 5 0 0 0
% -1
\put(16.1000,-34.8000){\makebox(0,0){-1}}%
% STR 2 0 3 0 Black White
% 4 2620 3670 2620 3770 5 0 0 0
% q^f-1
\put(26.2000,-37.7000){\makebox(0,0){$q^f-1$}}%
% STR 2 0 3 0 Black White
% 4 3590 3690 3590 3790 5 0 0 0
% q^{2f}-1
\put(35.9000,-37.9000){\makebox(0,0){$q^{2f}-1$}}%
% STR 2 0 3 0 Black White
% 4 5330 3680 5330 3780 5 0 0 0
% q^{3f}-1
\put(53.3000,-37.8000){\makebox(0,0){$q^{3f}-1$}}%
% STR 2 0 3 0 Black White
% 4 2000 1210 2000 1310 5 0 0 0
% \varphi_{K_{er}/K}(t)
\put(20.0000,-13.1000){\makebox(0,0){$s$}}%
% VECTOR 2 0 3 0 Black White
% 2 1400 3610 6590 3610
% 
\special{pn 8}%
\special{pa 1400 3610}%
\special{pa 6590 3610}%
\special{fp}%
\special{sh 1}%
\special{pa 6590 3610}%
\special{pa 6524 3590}%
\special{pa 6538 3610}%
\special{pa 6524 3630}%
\special{pa 6590 3610}%
\special{fp}%
% LINE 2 0 3 0 Black White
% 4 5330 2420 6460 2270 6460 2270 6450 2250
% 
\special{pn 8}%
\special{pa 5330 2420}%
\special{pa 6460 2270}%
\special{fp}%
\special{pa 6460 2270}%
\special{pa 6450 2250}%
\special{fp}%
% STR 2 0 3 0 Black White
% 4 6700 3510 6700 3610 5 0 0 0
% t
\put(67.0000,-36.1000){\makebox(0,0){t}}%
% LINE 2 2 3 0 Black White
% 4 3570 2800 3570 3610 3570 3610 3570 3610
% 
\special{pn 8}%
\special{pa 3570 2800}%
\special{pa 3570 3610}%
\special{dt 0.045}%
\special{pa 3570 3610}%
\special{pa 3570 3610}%
\special{dt 0.045}%
\end{picture}%

\end{center}
then $\delta_K$ induces the isomorphism
$$
 (1+\frak{p}_K^k)/(1+\frak{p}_K^{er})\,\tilde{\to}\,
 \text{\rm Gal}(K^{(er)}K^{\text{\rm ur}}/K^{(k)}K^{\text{\rm ur}})
 =V^s
$$
for $k-1<s\leq k$ ($k=1,2,\cdots$), and hence, for 
$V_t=V_t(K^{(er)}F^{\text{\rm ur}}/F^{\text{\rm ur}})$, we have
$$
 |V_t|=\begin{cases}
        e\cdot q^{nr}(1-q^{-f})&:t=0,\\
        q^{nr-fk}              &:q^{f(k-1)}-1<t\leq q^{fk}-1.
       \end{cases}
$$
The explicit actions 
\eqref{eq:explicit-action-of-theta-tau-0} and 
\eqref{eq:action-of-relative-weil-group-on-induced-rep} combined with 
Proposition
\ref{prop:conductor-and-conjugate-triviality-of-tilde-vartheta}
 shows that the space $\widehat{\frak g}^{V_t}$ 
of the $\text{\rm Ad}\circ\varphi(V_t)$-fixed
vectors in $\widehat{\frak g}$ is
$$
 \left\{\begin{bmatrix}
         a_11_e&      &      &      \\
               &a_21_e&      &      \\
               &      &\ddots&      \\
               &      &      &a_f1_e
        \end{bmatrix}\biggm|
           \begin{array}{l}
            a_i\in\Bbb C,\\
            a_1+a_2+\cdots+a_f=0
           \end{array}\right\}
$$
if $t=0$, 
$$
   \left\{\begin{bmatrix}
           a_1&   &      &   \\
              &a_2&      &   \\
              &   &\ddots&   \\
              &   &      &a_n
          \end{bmatrix}\biggm|
           \begin{array}{l}
            a_i\in\Bbb C,\\
            a_1+a_2+\cdots+a_n=0
           \end{array}\right\}
$$
if $0<t\leq q^{f\{e(r-1)-1\}}-1$, 
$$
    \left\{\begin{bmatrix}
           A_1&   &      &   \\
              &A_2&      &   \\
              &   &\ddots&   \\
              &   &      &A_f
          \end{bmatrix}\biggm|
           \begin{array}{l}
            A_i\in M_e(\Bbb C),\\
            \text{\rm tr}(A_1+A_2+\cdots+A_n)=0
           \end{array}\right\}
$$
if $q^{f\{e(r-1)-1\}}-1<t\leq q^{fe(r-1)}-1$ and 
$\widehat{\frak g}$ if $q^{fe(r-1)}-1<t\leq q^{fer}-1$. So we have
$$
 \dim_{\Bbb C}{\widehat{\frak g}}^{V_t}
  =\begin{cases}
   f-1   &:t=0,\\
   n-1   &:0<t\leq q^{f\{e(r-1)-1\}}-1,\\
   fe^2-1&:q^{f\{e(r-1)-1\}}-1<t\leq q^{fe(r-1)}-1,\\
   n^2-1 &:q^{fe(r-1)}-1<t\leq q^{fer}-1
  \end{cases}.
$$
Hence we have
\begin{align*}
 \sum_{t=0}^{\infty}&(V_0:V_t)^{-1}
       \dim_{\Bbb C}\left(\widehat{\frak g}/{\widehat{\frak g}}^{V_t}
                          \right)\\
 &=n^2-f+(n^2-n)\cdot\frac 1e\cdot\{e(r-1)-1\}+(n^2-fe^2)\cdot\frac 1e\\
 &=rn(n-1).
\end{align*}
Combined with
\eqref{eq:ratio-of-l-function-of-tamely-induced-rep-of-weil-group}, we
have
\begin{equation}
 \gamma(\varphi,\text{\rm Ad},\psi,d(x),0)
 =w(\text{\rm Ad}\circ\varphi)\cdot
  q^{rn(n-1)/2}\cdot f\cdot\frac{1-q^{-1}}
                                {1-q^{-f}}
\label{eq:explicit-value-of-gamma-factor}
\end{equation}
where $d(x)$ is the Haar measure on $F$ such that the volume of $O_F$
is one. 
The explicit value of the root number $w(\text{\rm Ad}\circ\varphi)$
will be given in the subsection 
\ref{subsec:root-number-of-adjoint-representation}.

\subsection{$\gamma$-factor of principal parameter}
\label{subsec:gamma-factor-of-principal-parameter}
Let $\text{\rm Sym}_{n-1}$ be the symmetric tensor representation of
$SL_2(\Bbb C)$ on the space of the complex coefficient homogeneous
polynomials of $X,Y$ of degree $n-1$, which gives the group
homomorphism 
$$
 \text{\rm Sym}_{n-1}:SL_2(\Bbb C)\to GL_n(\Bbb C)
$$
with respect to the $\Bbb C$-basis 
$\left\{v_k=\frac 1{(k-1)!}X^{n-k}Y^{k-1}\right\}_{k=1,2,\cdots,n}$. 
Then 
$$
 d\text{\rm Sym}_{n-1}\begin{bmatrix}
                       0&1\\
                       0&0
                      \end{bmatrix}=N_0
 =\begin{bmatrix}
   0&1& &      & \\
    &0&1&      & \\
    & &\ddots&\ddots& \\
    & & &0&1\\
    & & &      &0
   \end{bmatrix}
$$
is the nilpotent element in $\frak{pgl}_n(\Bbb C)=\frak{sl}_n(\Bbb C)$
associated with the standard {\it \'epinglage} of the standard root
system of $\frak{sl}_n(\Bbb C)$. Then 
\begin{equation}
 \varphi_0:W_F\times SL_2(\Bbb C)
           \xrightarrow{\text{\rm proj.}}SL_2(\Bbb C)
           \xrightarrow{\text{\rm Sym}_{n-1}}GL_n(\Bbb C)
           \xrightarrow{\text{\rm canonical}}PGL_n(\Bbb C)
\label{eq:candidate-of-principal-parameter}
\end{equation}
is a representation of Weil-Deligne group with the associated triplet 
$(\rho_0,PGL_n(\Bbb C),N_0)$ such that
$\rho_0|_{I_F}$ is trivial and
$$
 \rho_0(\widetilde{\text{\rm Fr}})
 =\overline{
   \begin{bmatrix}
    q^{-(n-1)/2}&            &      &           &           \\
                &q^{-(n-3)/2}&      &           &           \\
                &            &\ddots&           &           \\
                &            &      &q^{(n-3)/2}&           \\
                &            &      &           &q^{(n-1)/2}
   \end{bmatrix}}
 \in PGL_n(\Bbb C).
$$
Let $\text{\rm Ad}:PGL_n(\Bbb C)\to GL_{\Bbb C}(\widehat{\frak g})$ be 
the adjoint representation of $PGL_n(\Bbb C)$ on 
$\widehat{\frak g}=\frak{sl}_n(\Bbb C)$. Then 
\begin{equation}
 \left\{N_0^k\mid k=1,2,\cdots,n-1\right\}
\label{eq:c-basis-of-widehat-g-n-0-for-sl-n}
\end{equation}
is the $\Bbb C$-basis of 
$$
 \widehat{\frak g}_{N_0}
 =\{X\in\widehat{\frak g}\mid[X,N_0]=0\}.
$$ 
The representation
matrix of $\text{\rm Ad}\circ\rho_0(\widetilde{\text{\rm Fr}})$ on 
$\widehat{\frak g}_{N_0}$ with respect to the $\Bbb C$-basis 
\eqref{eq:c-basis-of-widehat-g-n-0-for-sl-n} is
$$
  \begin{bmatrix}
   q^{-1}&      &      &          \\
         &q^{-2}&      &          \\
         &      &\ddots&          \\
         &      &      &q^{-(n-1)}
  \end{bmatrix}
$$
so that we have
$$
 L(\varphi_0,\text{\rm Ad},s)
 =\prod_{k=1}^{n-1}\left(1-q^{-(s+k)}\right)^{-1}.
$$
On the other hand \cite[p.448]{Gross-Reeder2010} shows 
$$
 \varepsilon(\varphi_0,\text{\rm Ad},\psi,d(x),0)=q^{n(n-1)/2}
$$
where $d(x)$ is the Haar measure on $F$ such that the volume of $O_F$
is one. 
Since the symmetric tensor representation $\text{\rm Sym}_{n-1}$ is
self-dual, we have
\begin{equation}
 \gamma(\varphi_0,\text{\rm Ad},\psi,d(x),0)
 =q^{n(n-1)/2}\cdot\frac{1-q^{-1}}
                        {1-q^{-n}}.
\label{eq:gamma-factor-of-principal-parametor}
\end{equation}

\subsection{Verification of formal degree conjecture}
\label{subse:verification-of-formal-degree-conjecture}
Let $\mathcal{A}_{\varphi}$ be the centralizer of $\text{\rm Im}\,\varphi$
in $PGL_n(\Bbb C)$. 
Let us denote by $\mathcal{A}_{\widetilde\vartheta}$ the set of the group homomorphism 
$\lambda:W_{K/F}\to\Bbb C^{\times}$ whose restriction to $K^{\times}$
is a character $x\mapsto\widetilde\vartheta(x^{\tau-1})$ with some 
$\tau\in\text{\rm Gal}(K/F)$ which is uniquely determined by
  $\lambda$ due to Proposition 
\ref{prop:conductor-and-conjugate-triviality-of-tilde-vartheta}. 
Let us call it associated with $\lambda$. 
If $\tau\in\text{\rm Gal}(K/F)$ is associated with 
$\lambda\in\mathcal{A}_{\widetilde\vartheta}$, then we have 
\begin{equation}
 \widetilde\vartheta(x^{\sigma(\tau-1)})=\widetilde\vartheta(x^{\tau-1})
\label{eq:theta-x-tau-minus-1-is-stable-under-galois-action}
\end{equation}
for all $x\in K^{\times}$ and 
$\sigma,\tau\in\text{\rm Gal}(K/F)$, because 
$$
 \lambda(x^{\sigma})=\lambda((\sigma,1)^{-1}(1,x)(\sigma,1))
 =\lambda(x).
$$
This implies that $\mathcal{A}_{\widetilde\vartheta}$ is in fact a subgroup of the
character group of $W_{K/F}$. 

Take a $T\npmod{\Bbb C^{\times}}\in A_{\varphi}$ with 
$T\in GL_{\Bbb C}(V_{\vartheta})=GL_n(\Bbb C)$. Then
we have a character 
$$
 \lambda:W_{K/F}\to\Bbb C^{\times}
$$
such that $gT=\lambda(g)Tg$ for all $g\in W_{K/F}$. If 
$(Tv_1)(\tau)\neq 0$ with $\tau\in\text{\rm Gal}(K/F)$ then 
$x\cdot T(v_1)=\lambda(x)T(x\cdot v_1)$ for
$x\in K^{\times}$ implies 
$\widetilde\vartheta(x^{\tau})=\lambda(x)\widetilde\vartheta(x)$
for all $x\in K^{\times}$. So $(Tv_1)(\tau^{\prime})\neq 0$ for some
$\tau^{\prime}\in\text{\rm Gal}(K/F)$ implies that 
$\widetilde\vartheta(x^{\tau})=\widetilde\vartheta(x^{\tau^{\prime}})$
for all $x\in K^{\times}$ and hence $\tau^{\prime}=\tau$, due to
Proposition
\ref{prop:conductor-and-conjugate-triviality-of-tilde-vartheta}. 
Hence we have $Tv_1=c\cdot v_{\tau}$
with $c\in\Bbb C^{\times}$. Then we have
$$
 Tv_{\sigma}=c\cdot\lambda(\sigma)^{-1}\sigma\cdot v_{\tau}
 =c\cdot\lambda(\sigma)^{-1}\widetilde\vartheta\left(\alpha_{K/F}(\sigma,\tau)\right)
  \cdot v_{\sigma\tau}
$$
for all $\sigma\in\text{\rm Gal}(K/F)$. We have 

\begin{prop}
\label{prop:centralizer-of-image-of-varphi-in-pgl(v)-general-case}
$\overline T\mapsto\lambda$ gives a group isomorphism of
  $A_{\varphi}$ onto $\mathcal{A}_{\widetilde\vartheta}$.
\end{prop}
\begin{proof}
It is clear that $\overline T\mapsto\lambda$ is injective group
homomorphism, because $\text{\rm Ind}_{K^{\times}}^{W_{K/F}}\widetilde\vartheta$ is
irreducible. 
Take any $\lambda\in\mathcal{A}_{\widetilde\vartheta}$ and the 
$\tau\in\text{\rm Gal}(K/F)$ associated with it. 
Define a $T\in GL_{\Bbb C}(V)$ by
$$
 Tv_{\sigma}
 =\lambda(\sigma)^{-1}\widetilde\vartheta\left(\alpha_{K/F}(\sigma,\tau)\right)
   \cdot v_{\sigma\tau}
$$
for all $\sigma\in\text{\rm Gal}(K/F)$. Then we have 
$gT=\lambda(g)\cdot Tg$ for all $g\in W_{K/F}$. 
\end{proof}

We have in fact

\begin{prop}\label{prop:a-theta-is-character-group-of-gal-k-over-f}
$\mathcal{A}_{\widetilde\vartheta}$ is equal to the group of the
character $\lambda$ of $W_{K/F}$ which is trivial on $K^{\times}$. In 
particular
\begin{equation}
 |\mathcal{A}_{\widetilde\vartheta}|
 =(O_K:N_{K/F}(O_K^{\times}))\cdot f.
\label{eq:order-of-centralizer-in-pgl}
\end{equation}
\end{prop}
\begin{proof}
Let $\lambda:W_{K/F}\to\Bbb C^{\times}$ be a group homomorphism such
that $\lambda(x)=\widetilde\vartheta(x^{\tau-1})$ for all 
$x\in K^{\times}$ with some $\tau\in\text{\rm Gal}(K/F)$. Because of 
\eqref{eq:theta-x-tau-minus-1-is-stable-under-galois-action}, We have 
$$
 \widetilde\vartheta(x^{\tau-1})^n
 =\prod_{\sigma\in\text{\rm Gal}(K/F)}
   \widetilde\vartheta(x^{\sigma(\tau-1)})=1
$$
for all $x\in K^{\times}$. Since $n$ is prime to $p$, the mapping 
$x\mapsto x^n$ is surjective group homomorphism of $1+\frak{p}_K$ onto 
$1+\frak{p}_K$. Hence 
$1+\frak{p}_K\subset\text{\rm Ker}(\lambda|_{O_K^{\times}})$, that is 
$\widetilde\vartheta(x^{\tau})=\widetilde\vartheta(x)$ for
all $x\in 1+\frak{p}_K$. Then $\tau=1$ by Proposition 
\ref{prop:conductor-and-conjugate-triviality-of-tilde-vartheta}. 
So $\lambda$ is trivial on $K^{\times}$. Now we have 
$$
 |\mathcal{A}_{\widetilde\vartheta}|=(K_1:F)=(K_1:K_0)\cdot f
$$
where $K_1$ is the maximal abelian subextension of $K/F$. On the other
hand we have $N_{K/F}(O_K^{\times})=N_{K_1/F}(O_{K_1}^{\times})$, and
hence we have
$$
 (O_F^{\times}:N_{K/F}(O_K^{\times})=e(K_1/F)=(K_1:K_0)
$$
because $K_0$ is the maximal unramified subextension of $K/F$.
\end{proof}

Let $d_{G(F)}(g)$ be the Haar measure on $G(F)=SL_n(F)$ with respect
to which the volume of $G(O_F)=SL_n(O_F)$ is one. Then the 
Euler-Poincar\'e measure $\mu_{G(F)}$ on $G(F)$ is
$$
 d\mu_{G(F)}(g)
 =(-1)^{n-1}q^{n(n-1)/2}\prod_{k=1}^{n-1}(1-q^{-k})\cdot 
  d_{G(F)}(g)
$$
(see \cite[3.4, Theor\'eme 7]{Serre1971}). Then, by Theorem 
\ref{th:supercuspidal-representation-of-sl(n)}, the formal degree of 
the supercuspidal representation 
$\pi_{\beta,\theta}
 =\text{\rm ind}_{G(O_F)}^{G(F)}\delta_{\beta,\theta}$ with respect to
 the absolute value of the Euler-Poincar\'e measure on $G(F)$ is
$$
 \frac{q^{(r-1)n(n-1)/2}}
      {(O_F^{\times}:N_{K/F}(O_K^{\times}))}\cdot
 \frac{1-q^{-n}}
      {1-q^{-f}}.
$$
Now the formulae 
\eqref{eq:explicit-value-of-gamma-factor}, 
\eqref{eq:gamma-factor-of-principal-parametor} and 
\eqref{eq:order-of-centralizer-in-pgl} give the following

\begin{thm}\label{th:formal-degree-conjecture-for-sl(n)}
The formal degree of 
the supercuspidal representation 
$\pi_{\beta,\theta}
 =\text{\rm ind}_{G(O_F)}^{G(F)}\delta_{\beta,\theta}$ with respect to
 the absolute value of the Euler-Poincar\'e measure on $G(F)$ is
$$
 \frac 1{|\mathcal{A}_{\varphi}|}\cdot
 \left|\frac{\gamma(\varphi,\text{\rm Ad},\psi,d(x),0)}
            {\gamma(\varphi_0,\text{\rm Ad},\psi,d(x),0)}\right|
$$
where $d(x)$ is the Haar measure on $F$ such that the volume of $O_F$
is one. 
\end{thm}

Since $\mathcal{A}_{\varphi}$ is a finite abelian group, all the
irreducible representation of $\mathcal{A}_{\varphi}$ is
one-dimensional. So Theorem 
\ref{th:formal-degree-conjecture-for-sl(n)} says that the formal
degree conjecture is valid if we consider 
\eqref{eq:candidate-of-langlands-arthur-parameter} as the
Arthur-Langlands parameter of the supercuspidal representation
$\pi_{\beta,\theta}$ and 
\eqref{eq:candidate-of-principal-parameter} as the principal parameter
of $G(F)=SL_n(F)$.

\section{Root number conjecture}
\label{sec:root-number-conjecture}
In this section, we will assume that $K/F$ is a tamely ramified Galois
extension such that the degree $(K:F)=n$ is prime to $p$, and will
keep the notations of preceding sections. Put 
$\Gamma=\text{\rm Gal}(K/F)$. 

%\subsection{Structure of adjoint representation}
%\label{subsec:structure-of-adjoint-representation}

\subsection{Root number of adjoint representation}
\label{subsec:root-number-of-adjoint-representation}
We will identify the representations of $W_{K/F}$ with the
representations of $W_F$ which factor through the canonical
surjection 
$$
 W_F\to W_F/\overline{[W_K.W_K]}=W_{K/F}.
$$
We will also regard a representation of $\Gamma$ as the representation
of $W_{K/F}$ via the projection $W_{K/F}\to\Gamma$.
Then we have

\begin{prop}\label{structure-of-adjoint-representation}
$$
 \text{\rm Ad}\circ\varphi
 =\bigoplus_{1\neq\pi\in\Gamma\sphat}\pi^{\dim\pi}
  \oplus
  \bigoplus_{1\neq\gamma\in\Gamma}
   \text{\rm Ind}_{K^{\times}}^{W_{K/F}}
    \widetilde\vartheta_{\gamma}
$$
where $\Gamma\sphat$ is the set of the equivalence classes of the
irreducible complex representations of $\Gamma$, and 
$\widetilde\vartheta_{\gamma}(x)=\widetilde\vartheta(x^{1-\gamma})$ 
($x\in K^{\times}$) for $1\neq\gamma\in\Gamma$.
\end{prop}
\begin{proof}
The adjoint action of $\varphi_1$ on 
$\frak{gl}_n(\Bbb C)=\Bbb C\oplus\frak{sl}_n(\Bbb C)$ gives
$$
 \left(\text{\rm $\text{\rm Ad}\circ\varphi$ on $\widehat{\frak{g}}$}
      \right)\oplus\text{\rm trivial representation of $W_F$}
 =\varphi_1\otimes\varphi_1^{\ast}
$$
where $\varphi_1^{\ast}$ is the dual representation of $\varphi_1$. So
the character of $\text{\rm Ad}\circ\varphi$ is
$$
 \chi_{\text{\rm Ad}\circ\varphi}(g)
 =\chi_{\varphi_1}(g)\cdot\overline{\chi_{\varphi_1}(g)}
 =\begin{cases}
   -1&:\sigma\neq 1,\\
   n-1+\sum_{\tau,\gamma\in\Gamma, \tau\neq\gamma}
        \widetilde\vartheta(x^{\tau-\gamma})
     &:\sigma=1.
  \end{cases}
$$
for 
$g=(\sigma,x)\in W_{K/F}=\Gamma{\ltimes}_{\alpha_K/F}K^{\times}$. Since
$$
 \sum_{\tau,\gamma\in\Gamma, \tau\neq\gamma}
         \widetilde\vartheta(x^{\tau-\gamma})
 =\sum_{1\neq\gamma\in\Gamma}\sum_{\tau\in\Gamma}
         \widetilde\vartheta(x^{\tau(\gamma-1)}),
$$
we have the required decomposition of $\text{\rm Ad}\circ\varphi$. 
\end{proof}

In order to calculate the $\varepsilon$-factor of 
$\text{\rm Ad}\circ\varphi$, we will fix  the additive
character 
$$
 \psi_F:F\xrightarrow{T_{F/\Bbb Q_p}}
        \Bbb Q_p\xrightarrow{\text{\rm canonical}}
        \Bbb Q_p/\Bbb Z_p\,\tilde{\to}\,
        \Bbb Q/\Bbb Z\xrightarrow{\exp(2\pi\sqrt{-1}\ast)}
        \Bbb C^{\times}
% \{x\in F\mid\psi_F(xO_F)=1\}=\mathcal{D}(F/\Bbb Q_p)^{-1}
%                             =\frak{p}_F^{-d(F)}
$$
of $F$ so that
$$
 \{x\in F\mid\varphi_F(xO_F)=1\}=\mathcal{D}(F/\Bbb Q_p)^{-1}
                                =\frak{p}_F^{-d(F)},
$$
and the Haar measure $d_F(x)$ on $F$ such that 
$$
 \int_{O_F}d_F(x)=q^{-d(F)}.
$$
Then $\psi_L=\psi_F\circ T_{L/F}$ for any finite extension $L/F$. 
For the sake of simplicity, put
$$
 \varepsilon(\ast,\psi_F)
 =\varepsilon(\ast,\psi_F,d_F(x),0).
$$
On the other hand the additive character $\psi$ of $F$ is such that
$$
 \{x\in F\mid\psi(xO_F)=1\}=O_F
$$
and the Haar measure $d(x)$ on $F$ is such that $\int_{O_F}d(x)=1$.
Then we have

\begin{thm}\label{th:epsilon-factor-of-ad-varphi}
Assume that $r\geq 3$. Then we have
$$
 \varepsilon(\varphi,\text{\rm Ad},\psi,d(x),0)
 =w(\varphi,\text{\rm Ad})\cdot q^{rn(n-1)/2}
$$
with
$$
 w(\varphi,\text{\rm Ad})
 =\vartheta((-1)^{n-1})\times
   \begin{cases}
    (-1)^{\frac{q-1}2\cdot f}&:\text{\rm $e$=even},\\
    1&:\text{\rm $e$=odd}.
  \end{cases}
$$
\end{thm}
\begin{proof}
Note that
$$
 \bigoplus_{1\neq\pi\in\Gamma\sphat}\pi^{\dim\pi}
 =\text{\rm Ind}_K^F\text{\bf 1}_K-\text{\bf 1}_F
 \;\;\text{\rm and}\;\;
 \text{\rm Ind}_{K^{\times}}^{W_{K/F}}\widetilde\vartheta_{\gamma}
 =\text{\rm Ind}_{K^{\times}}^{W_{K/F}}\widetilde\vartheta_{\gamma^{-1}}
$$
for $1\neq\gamma\in\Gamma$. So, if we choose a subset 
$S\subset\Gamma$ such that
$$
 \{\gamma\in\Gamma\mid\gamma^2\neq 1\}
 =\{\gamma, \gamma^{-1}\mid\gamma\in S\}
 \;\;\text{\rm and}\;\;
 S\cap S~{-1}=\emptyset,
$$
then we have
\begin{align}
 \varepsilon(\text{\rm Ad}\circ\varphi,\psi_F)
 =&\lambda(K/F,\psi_F)^n\times
   \varepsilon(\text{\bf 1}_K,\psi_K)\cdot
   \varepsilon(\text{\bf 1}_F,\psi_F)^{-1} \nonumber\\
 &\times\prod_{\gamma\in S}
   \varepsilon(\widetilde\vartheta_{\gamma},\psi_F)\cdot
   \varepsilon(\widetilde\vartheta_{\gamma^{-1}},\psi_F)
  \times\prod_{1\neq\gamma\in\Gamma, \gamma^2=1}
   \varepsilon(\widetilde\vartheta_{\gamma},\psi_F).
\label{eq:decomposition-of-epsilon-factor-into-abelian}
\end{align}
By \eqref{eq:explicit-value-of-epsilon-factor-trivial}, we have
\begin{align}
 \varepsilon(\text{\bf 1}_K,\psi_K)\cdot
 \varepsilon(\text{\bf 1}_F,\psi_F)^{-1}
 &=q^{fd(K)/2}\cdot q^{-d(F)/2} \nonumber\\
 &=q^{(n-1)d(F)/2+(n-f)/2},
\label{eq:explicit-value-of-epsilon-factor-regular-rep-part}
\end{align}
since $K/F$ is tamely ramified and hence 
$\mathcal{D}(K/F)=\frak{p}_K^{e-1}$. For the Gauss sum, we have
$$
 G_{\psi_{K_0}}\left(\left(\frac{\ast}
                                {K_0}\right),
                     \varpi_0^{-(d(K_0)+1)}\right)^2
 =\left(\frac{-1}
             {K_0}\right)
 =(-1)^{(q^f-1)/2}.
$$
Then, by Proposition
\ref{prop:lambda-factor-of-tamely-ramified-galois-ext}, we have
\begin{equation}
 \lambda(K/F,\psi_F)
 =\begin{cases}
   (-1)^{\frac{q-1}2\cdot f}
       &:\text{\rm $e$=even},\\
   1   &:\text{\rm $e$=odd}.
  \end{cases}
\label{eq:explicit-value-of-lambda-k-over-f-n-power}
\end{equation}
Note that if $e$ is even, since $e$ divides $q^f-1$, we have
$$
 \frac{q^f-1}2\cdot\frac e2
 \equiv\frac{q^f-1}2\npmod{2}
 \equiv\frac{q-1}2\cdot f\npmod{2}.
$$
For any $1\neq\gamma\in\Gamma$, we have
$$
 n(\widetilde\vartheta)
 =\text{\rm Min}\{0<k\in\Bbb Z\mid
        \widetilde\vartheta_{\gamma}(1+\frak{p}_K^k)=1\}
 =\begin{cases}
    e(r-1)+1&:\gamma\not\in\text{\rm Gal}(K/K_0),\\
    e(r-1)  &:\gamma\in\text{\rm Gal}(K/K_0)
   \end{cases}
$$
by Proposition
\ref{prop:conductor-and-conjugate-triviality-of-tilde-vartheta}. So we
have
\begin{align}
 &\prod_{\gamma\in S}
   \varepsilon(\widetilde\vartheta_{\gamma},\psi_K)\cdot
   \varepsilon(\widetilde\vartheta_{\gamma}^{-1},\psi_K)\times
  \prod_{1\neq\gamma\in\Gamma : \gamma^2=1}
    \varepsilon(\widetilde\vartheta_{\gamma},\psi_K) \nonumber\\
 =&q_K^{(e-1)\{d(K)+e(r-1)\}/2}\times q_K^{(n-e)\{d(K)+e(r-1)+1\}/2}
                                                     \nonumber\\
  &\times\prod_{\gamma\in S}\left\{
    \begin{array}{l}
     G_{\psi_K}(\widetilde\vartheta_{\gamma}^{-1},
                -\varpi_K^{-(d(K)+n(\widetilde\vartheta_{\gamma}))})\cdot
     \widetilde\vartheta_{\gamma}
                (\varpi_K)^{d(K)+n(\widetilde\vartheta_{\gamma})}
                                                     \nonumber\\
     \times
     G_{\psi_K}(\widetilde\vartheta_{\gamma},
                -\varpi_K^{-(d(K)+n(\widetilde\vartheta_{\gamma}))})\cdot
     \widetilde\vartheta_{\gamma}
               (\varpi_K)^{-(d(K)+n(\widetilde\vartheta_{\gamma}))}
    \end{array}\right\} \nonumber\\
  &\times\prod_{1\neq\gamma\in\Gamma : \gamma^2=1}
    G_{\psi_K}(\widetilde\vartheta_{\gamma}^{-1},
                -\varpi_K^{-(d(K)+n(\widetilde\vartheta_{\gamma}))})\cdot
     \widetilde\vartheta_{\gamma}
               (\varpi_K)^{d(K)+n(\widetilde\vartheta_{\gamma})}
                                                     \nonumber\\
 =&q^{(d(F)+r)n(n-1)/2-(n-f)/2} \nonumber\\
  &\times\prod_{1\neq\gamma\in\Gamma : \gamma^2=1}
    G_{\psi_K}(\widetilde\vartheta_{\gamma}^{-1},
                -\varpi_K^{-(d(K)+n(\widetilde\vartheta_{\gamma}))})\cdot
     \widetilde\vartheta_{\gamma}
              (\varpi_K)^{d(K)+n(\widetilde\vartheta_{\gamma})}.
\label{eq:explicit-value-of-epsilon-factor-of-vartheta-gamma}
\end{align}
For any $1\neq\gamma\in\Gamma$ such that $\gamma^2=1$, put 
$K=K_{\gamma}(\xi_{\gamma})$ where $\xi_{\gamma}\in K$ such that 
$\xi_{\gamma}^2\in K_{\gamma}$. Then 
\cite[Th.3]{Frohlich-Queyrut1973} shows that
$$
 G_{\psi_K}(\widetilde\vartheta_{\gamma}^{-1},
             \varpi_K^{-(d(K)+n(\widetilde\vartheta_{\gamma}))})
 \cdot
 \widetilde\vartheta_{\gamma}
       (\varpi_K)^{d(K)+n(\widetilde\vartheta_{\gamma})}
 =\widetilde\vartheta_{\gamma}(\xi_{\gamma})
  =\widetilde\vartheta(-1).
$$
On the other hand we have
\begin{align*}
  \sharp\{1\neq\gamma\in\Gamma\mid\gamma^2=1\}
 =&\sharp\{H\subset\Gamma:\text{\rm subgroup s.t. $|H|=2$}\}\\
 &\begin{cases}
   \equiv 1\npmod{2}&:\text{\rm $n$ is even,}\\
   =0&:\text{\rm $n$ is odd}.
  \end{cases}
\end{align*}
Then we have
\begin{equation}
 \prod_{1\neq\gamma\in\Gamma:\gamma^2=1}
  G_{\psi_K}(\widetilde\vartheta_{\gamma}^{-1},
             \varpi_K^{-(d(K)+n(\widetilde\vartheta_{\gamma}))})
 =\widetilde\vartheta(-1)^{n-1}
 =\vartheta((-1)^{n-1})
\label{eq:explicit-value-of-product-of-gauss-sum-of-order-two-gamma}
\end{equation}
Combining 
\eqref{eq:decomposition-of-epsilon-factor-into-abelian}, 
\eqref{eq:explicit-value-of-epsilon-factor-regular-rep-part}, 
\eqref{eq:explicit-value-of-lambda-k-over-f-n-power}, 
\eqref{eq:explicit-value-of-epsilon-factor-of-vartheta-gamma} and
\eqref{eq:explicit-value-of-product-of-gauss-sum-of-order-two-gamma}, 
we have
\begin{align*}
 \varepsilon(\text{\rm Ad}\circ\varphi,\psi_F)
 =&q^{d(F)(n^2-1)/2+rn(n-1)/2}\cdot\vartheta((-1)^{n-1})\\
  &\times\begin{cases}
          (-1)^{\frac{q-1}2\cdot f}&:\text{\rm $e$ is even},\\
          1&:\text{\rm $e$ is odd}.
         \end{cases}
\end{align*}
Since $\psi_F(x)=\psi(\varpi_F^{d(f)}x)$ and $d_F(x)=q^{-d(F)/2}\cdot
d(x)$, we have the required formula of 
$\varepsilon(\text{\rm Ad}\circ\varphi,\psi,d(x),0)$ by 
Proposition \ref{prop:scaling-o-epsilon-factor}.
\end{proof}

\subsection{Verification of root number conjecture}
\label{subsec:verification-of-root-number-onjecture}
Let $D$ be the maximal torus of $SL_n$ consisting of the diagonal
matrices. 
%The group $X(D)$ of the character of $D$ is identified with 
%$\Bbb Z^n/\Bbb Z(1,\cdots,1)$ by $\dot l\mapsto\alpha_l$ where
%$\alpha_l(t)=\prod_{i=1}^nt_i^{l_i}$ for
%$$
% t=\begin{bmatrix}
%    t_1&      &  \\
%       &\ddots&  \\
%       &      &t_n
%   \end{bmatrix}
%  \in D.
%$$
The group $X^{\vee}(D)$ of the one-parameter subgroup of $D$ is
identified with
$$
 \Bbb Z^n_{\text{\rm tr}=0}
 =\{(m_1,\cdots,m_n)\in\Bbb Z^n\mid m_1+\cdots+m_n=0\}
$$
by $m\mapsto u_m$ where
$$
 u_m(t)=\begin{bmatrix}
         t^{m_1}&      &       \\
                &\ddots&       \\
                &      &t^{m_n}
        \end{bmatrix},
$$
or we will denote by $u_m=\sum_{i=1}^nm_iu_i$. Then the set 
of the co-roots of $SL_n$ with respect to $D$ is
$$
 \Phi^{\vee}(D)
 =\{u_i-u_j\in X^{\vee}(D)\mid 1\leq i,j\leq n,\;i\neq j\}.
$$
Now we have
$$
 2\cdot\rho=\sum_{1\leq i<j\leq n}\left(u_i-u_j\right)
           =\sum_{i=1}^n(n+1-2i)u_i.
$$
So the special central element is
$$
 \epsilon=2\cdot\rho(-1)=(-1)^{n-1}1_n\in SL_n(F).
$$
Since
$$
 \pi_{\beta,\theta}
 =\text{\rm ind}_{G(O_F)}^{G(F)}\delta_{\beta,\theta},
 \qquad
 \delta_{\beta,\theta}
 =\text{\rm Ind}_{G(O_F/\frak{p}_F^r;\beta)}^{G(O_F/\frak{p}_F^r)}
   \sigma_{\beta,\theta},
$$
by recalling the construction of $\sigma_{\beta,\theta}$, we have
$$
 \pi_{\beta,\theta}(\epsilon)
 =\delta_{\beta,\theta}(\epsilon)
 =\sigma_{\beta,\theta}(\epsilon)
 =\theta((-1)^{n-1}).
$$
On the other hand, since $\vartheta=c\cdot\theta$, we have
$$
 w(\varpi,\text{\rm Ad})=\theta((-1)^{n-1})
$$
by Proposition \ref{prop:explicit-value-of-c((-1)(n-1))} and 
Theorem \ref{th:epsilon-factor-of-ad-varphi}. 
So we have proved the following theorem

\begin{thm}\label{th:root-number-conjecture}
$w(\varphi,\text{\rm Ad})=\pi_{\beta,\theta}(\epsilon)$.
\end{thm}

This theorem says that the root number  conjecture is valid if we
consider 
\eqref{eq:candidate-of-langlands-arthur-parameter} as the
Arthur-Langlands parameter of the supercuspidal representation
$\pi_{\beta,\theta}$. 

\appendix
\section{Local factors}
\label{sec:local-factors}
Fix an algebraic closure $F^{\text{\rm alg}}$ of $F$ in which we will
take every algebraic extensions of $F$. Put
$$
 \nu_F(x)=(F(x):F)^{-1}\text{\rm ord}_F(N_{F(x)/F}(x))
 \;\text{\rm for $\forall x\in F^{\text{\rm alg}}$}
$$
and
$$
 O_K=\{x\in F^{\text{\rm alg}}\mid\nu_F(x)\geq 0\},
 \quad
 \frak{p}_K=\{x\in F^{\text{\rm alg}}\mid\nu(F)(x)>0\}.
$$
Then $\Bbb K=O_K/\frak{p}_K$ is an algebraic extension of 
$\Bbb F=O_F/\frak{p}_F$. If $K/F$ is a finite extension, fix a
generator $\varpi_K\in O_K$ of $\frak{p}_K$.

\subsection{Weil group}
\label{subsec:weil-group}
Let $F^{\text{\rm ur}}$ be the maximal unramified extension of $F$ and 
$\text{\rm Fr}\in\text{\rm Gal}(F^{\text{\rm ur}}/F)$ the
inverse of the Frobenius automorphism of $F^{\text{\rm ur}}$ over
$F$. The the Weil group $W_F$ of $F$ is 
$$
 W_F=\left\{\sigma\in\text{\rm Gal}(F^{\text{\rm alg}}/F)\mid
              \sigma|_{F^{\text{\rm ur}}}\in
               \langle\text{\rm Fr}\rangle\right\}    
$$
The
group $W_F$ is a locally compact group with respect to the topology
such that $I_F=\text{\rm Gal}(F^{\text{\rm alg}}/F^{\text{\rm ur}})$
is an open compact subgroup of $W_F$. 

Let $F^{\text{\rm ab}}$ be the maximal abelian extension of $F$ in
$F^{\text{\rm alg}}$. Then
$$
 \overline{[W_F,W_F]}
 =\text{\rm Gal}(F^{\text{\rm alg}}/F^{\text{\rm ab}})
$$
and
$$
 W_F/\overline{[W_F,W_F]}\xrightarrow[\text{\rm res.}]{\sim}
 \{\sigma\in\text{\rm Gal}(F^{\text{\rm ab}}/F)\mid
    \sigma|_{F^{\text{\rm ur}}}\in\langle\text{\rm Fr}\rangle\}.
$$
So, by the local class field theory, there exists a topological group
isomorphism 
$$
 \delta_F:F^{\times}\,\tilde{\to}\,W_F/\overline{[W_F,W_F]}
$$
such that $\delta_F(\varpi)|_{F^{\text{\rm ur}}}=\text{\rm Fr}$. 
Fix a $\widetilde{\text{\rm Fr}}\in\text{\rm Gal}(F^{\text{\rm alg}}/F)$
such that 
$\widetilde{\text{\rm Fr}}|_{F^{\text{\rm ab}}}
 =\delta_F(\varpi)$. Then
$$
 W_F=\langle\widetilde{\text{\rm Fr}}\rangle\ltimes
      \text{\rm Gal}(F^{\text{\rm alg}}/F^{\text{\rm ur}}).
$$
Let $K/F$ be a finite extension in $F^{\text{\rm alg}}$. Then 
$K^{\text{\rm ur}}=K\cdot F^{\text{\rm ur}}$ and 
$$
 W_K=\{\sigma\in\text{\rm Gal}(F^{\text{\rm alg}}/K)\mid
        \sigma|_{F^{\text{\rm ur}}}\in\langle\text{\rm Fr}^f\rangle
          \}
    =\{\sigma\in W_F\mid\sigma|_K=1\},
$$
where $f=(\Bbb K:\Bbb F)$, is a closed subgroup of $W_F$. 
If further $K/F$ is a Galois extension, then 
$\overline{[W_K,W_K]}\triangleleft W_F$ and
$$
 W_{K/F}=W_F/\overline{[W_K,W_K]}
        =\{\sigma\in\text{\rm Gal}(K^{\text{\rm ab}}/F)\mid
            \sigma|_{F^{\text{\rm ur}}}\in\langle\text{\rm Fr}\rangle
             \}
$$
is called the relative Weil group of $K/F$. Then we have a exact
sequence 
$$
 1\to K{\times}\xrightarrow{\delta_K}W_{K/F}
               \xrightarrow{\text{\rm res.}}\text{\rm Gal}(K/F)
               \to 1
$$
which is the group extension associated with the fundamental calss 
$$
 [\alpha_{K/F}]\in H^2(\text{\rm Gal}(K/F),K^{\times}),
$$
that is, we can identify 
$W_{K/F}=\text{\rm Gal}(K/F)\times K^{\times}$ with the group
operation
$$
 (\sigma,x)\cdot(\tau,y)
 =(\sigma\tau,\alpha_{K/F}(\sigma,\tau)\cdot xy).
$$
Let $K_0=K\cap F^{\text{\rm ur}}$ be the maximal unramified
subextension of $K/F$. Then the fundamental calss can be chosen 
so that $\alpha_{K/F}(\sigma,\tau)\in O_K^{\times}$ for all 
$\sigma, \tau\in\text{\rm Gal}(K/K_0)$, and the image $I_{K/F}$ of 
$I_F=\text{\rm Gal}(F^{\text{\rm alg}}/F^{\text{\rm ur}})\subset W_F$
under the canonical surjection $W_F\to W_{K/F}$ is identified with 
$\text{\rm Gal}(K/K_0)\times O_K^{\times}$.

\subsection{Artin conductor of representations of Weil group}
\label{subsec:artin-conductor}
Let $(\Phi,V)$ be a finite dimensional continuous complex
representation of the Weil group $W_F$. Since 
$I_F\cap\text{\rm Ker}(\Phi)$ is an open subgroup of 
$I_F=\text{\rm Gal}(F^{\text{\rm alg}}/F^{\text{\rm ur}})$, there
exists a finite Galois extension $K/F^{\text{\rm ur}}$ such that 
$text{\rm Gal}(F^{\text{\rm alg}}/K)\subset\text{\rm Ker}(\Phi)$. Let 
\begin{align*}
 V_k=&V_k(K/F^{\text{\rm ur}})\\
 =&\left\{\sigma\in\text{\rm Gal}(K/F^{\text{\rm ur}})\mid
           x^{\sigma}\equiv x\npmod{\frak{p}_K^{k+1}}
           \,\text{\rm for}\,\forall x\in O_K\right\}
\end{align*}
be the $k$-th ramification group of $K/F^{\text{\rm ur}}$ put
$$
 \widetilde V_k
 =\left[\text{\rm Gal}(F^{\text{\rm alg}}/F^{\text{\rm ur}})
         \xrightarrow{\text{\rm res.}}
          \text{\rm Gal}(K/F^{\text{\rm ur}})\right]^{-1}V_k
$$
for $k=0,1,2,3,\cdots$. So $\widetilde V_0=I_F$. The Artin conductor 
$a(\Phi)=a(V)$ is defined by 
$$
 a(\Phi)=a(V)
 =\sum_{k=0}^{\infty}\dim_{\Bbb C}(V/V^{\Phi(\widetilde V_k)})
                     \cdot
                     |V_0/V_k|^{-1}
$$
where 
$$
 V^{\Phi(\widetilde V_k)}
 =\{v\in V\mid \Phi(\widetilde V_k)v=v\}
 \quad
 (k=0,1,2,3,\cdots).
$$

\subsection{$\varepsilon$-factor of representations of Weil group}
\label{subsec:epsilon-factor}
Fix a continuous unitary character $\psi:F\to\Bbb C^{\times}$ of the
additive group $F$ and a Haar measure $d(x)$ of $F$.

Langlands and Deligne \cite{Deligne1973} show that, for every
finite dimensional continuous complex representation $(\Phi,V)$ of
$W_F$, there exists a complex constant 
$$
 \varepsilon(\Phi,\psi,d(x))=\varepsilon(V,\psi,d(x))
$$
which satisfies the following relations:
\begin{enumerate}
\item an exact sequence 
$$
 1\to V^{\prime}\to V\to V^{\prime\prime}\to 1
$$
implies
$$
 \varepsilon(V,\psi,d(x))
 =\varepsilon(V^{\prime},\psi,d(x))\cdot
  \varepsilon(V^{\prime\prime},\psi,d(x)),
$$
\item for a positive real number $r$
$$
 \varepsilon(\Phi,\psi,r\cdot d(x))
 =r^{\dim\Phi}\cdot\varepsilon(\Phi,\psi,d(x)),
$$
\item for any finite extension $K/F$ and a finite dimensional
  continuous complex representation $\phi$ of $W_K$, we have
$$
 \varepsilon\left(\text{\rm Ind}_{W_K}^{W_F}\phi,\psi,d(x)\right)
 =\varepsilon\left(\phi,\psi\circ T_{K/F},d_K(x)\right)\cdot
  \lambda(K/F,\psi)^{\dim\phi}
$$
where $d_K(x)$ is a Haar measure of $K$ and
$$
 \lambda(K/F,\psi)=\lambda(K/F,\psi,d(x),d_K(x))
 =\frac{\varepsilon\left(\text{\rm Ind}_{W_K}^{W_F}\text{\bf 1}_K,
                         \psi,d(x)\right)}
       {\varepsilon\left(\text{\bf 1}_K,\psi\circ T_{K/F},d_K(x)
                          \right)},
$$
\item if $\dim\Phi=1$, then $\Phi$ factors through 
      $W_F/\overline{[W_F,W_F]}$ and put
$$
 \chi:F^{\times}\xrightarrow{\delta_F}
      W_F/\overline{[W_F,W_F]}\xrightarrow{\Phi}
      \Bbb C~{\times}.
$$
Then we have
$$
 \varepsilon(\Phi,\psi,d(x))=\varepsilon(\chi,\psi,d(x))
$$
where the right hand side is the $\varepsilon$-factor of Tate 
\cite{Tate1979}.
\end{enumerate}

If the Haar measure $d(x)$ of $F$ is normalized so that the Fourier
transform
$$
 \widehat\varphi(y)=\int_F\varphi(x)\cdot\psi(-xy)d(x)
$$
has inverse transform
$$
 \varphi(x)=\int_F\widehat\varphi(y)\cdot\psi(xy)d(y),
$$
in other words
$$
 \int_{O_F}d(x)=q^{-n(\psi)/2}\;\;\text{\rm with}\;\;
 \{x\in F\mid \psi(xO_F)=1\}=\frak{p}_F^{-n(\psi)},
$$
then the explicit value of the $\varepsilon$-factor
$\varepsilon(\chi,\psi,d(x))$ is  
\begin{enumerate}
\item if $\chi|_{O_F^{\times}}=1$, then 
\begin{equation}
 \varepsilon(\chi,\psi,d(x))
 =\chi(\varpi)^{n(\psi)}\cdot q^{n(\psi)/2},
\label{eq:explicit-value-of-epsilon-factor-trivial}
\end{equation}
\item if $\chi|_{O_F^{\times}}\neq 1$, then 
\begin{equation}
 \varepsilon(\chi,\psi,d(x))
 =G_{\psi}(\chi^{-1},-\varpi^{-(n(\psi)+n)})\cdot
  \chi(\varpi)^{n(\psi)+n}\cdot q^{-(n(\psi)+n)/2}
\label{eq:explicit-value-of-epsilon-facot-non-trivial}
\end{equation}
where 
$f(\chi)=\text{\rm Min}\{0<n\in\Bbb Z\mid\chi(1+\frak{p}_F^n)=1\}$ and
$$
 G_{\psi}(\chi^{-1},\varpi^{-(n(\psi)+f(\chi))})
 =q^{-n/2}\sum_{\dot t\in(O_F/\frak{p}_F^{f(\chi)})^{\times}}
           \chi(t)^{-1}\psi\left(-\varpi^{-(n(\psi)+f(\chi))}t\right)
$$
is the Gauss sum.
\end{enumerate}

\begin{rem}\label{remark:normalization-of-gauss-sum}
The definition of the Gauss sum is normalized so that 
$$
 \left|G_{\psi}(\chi^{-1},-\varpi^{-(n(\psi)+f(\chi))})\right|=1.
$$
\end{rem}

We have

\begin{prop}\label{prop:scaling-o-epsilon-factor}
\begin{enumerate}
\item Put $\psi_a(x)=\psi(ax)$ for $a\in F^{\times}$. Then
$$
 \varepsilon(\Phi,\psi_a,d(x))
 =\det\Phi(a)\cdot|a|_F^{-\dim\Phi}\cdot
  \varepsilon(\Phi,\psi,d(x))
$$
where
$$
 \det\Phi:F^{\times}\xrightarrow{\delta_F}
          W_F/\overline{[W_F,W_F]}\xrightarrow{\det\circ\Phi}
          \Bbb C^{\times}.
$$
\item For any $s\in\Bbb C$ 
\begin{align*}
 \varepsilon(\Phi,\psi,d(x),s)
 &=\varepsilon(\Phi\otimes|\cdot|_F^s,\psi,d(x))\\
 &=\varepsilon(\Phi,\psi,d(x))\cdot
   q^{-s(n(\psi)\cdot\dim\Phi+a(\Phi))}.
\end{align*}
\end{enumerate}
\end{prop}

\begin{prop}\label{prop:epsilon-factor-wrt-normaized-unramified-character}
If $n(\psi)=0$ and the Haar measure $d(x)$ is normalized so that 
$$
 \int_{O_F}d(x)=1,
$$
then
$$
 \varepsilon(\Phi,\psi,d(x))=w(\Phi)\cdot q^{a(\Phi)/2}
                            =w(V)\cdot q^{a(V)/2}
$$
with $w(\Phi)\in\Bbb C$ of absolute value one.
\end{prop}

\begin{prop}\label{prop:chain-relation-of-lambda-factor}
For finite extensions $F\subset K\subset L$, we have
$$
 \lambda(L/F,\psi)
 =\lambda(L/K,\psi\circ T_{K/F})\cdot
  \lambda(K/F,\psi)^{(L:K)}.
$$
\end{prop}

When $K/F$ is a finite tamely ramified Galois extension, the
maximal unramified subextension $K_0=K\cap F^{\text{\rm ur}}$ is a
cyclic extension of $F$ and $K/K_0$ is also cyclic extension. So, by
means of Proposition \ref{prop:chain-relation-of-lambda-factor}, we
can give the explicit value of $\lambda(K/F,\psi)$. 

Let $\psi_F:F\to\Bbb C^{\times}$ be a continuous unitary character
such that
$$
 \{x\in F\mid\psi_F(xOF)=1\}=\mathcal{D}(F/\Bbb Q_p)^{-1}
 =\frak{p}_F^{-d(F)}
$$
and the Haar measure $d_F(x)$ on $F$ is normalized so that
$$
 \int_{O_F}d_F(x)=q^{-d(F)}.
$$
Let $K/F$ be a tamely ramified finite Galois extension, and put
$\psi_K=\psi_F\circ T_{K/F}$. Put
$$
 e=e(K/F)=(K:K_0),
 \quad
 f=f(K/F)=(K_0:F)
$$
where $K_0=K\cap F^{\text{\rm ur}}$ is the maximal unramified
subextension of $K/F$. Let
$$
 \left(\frac{\varepsilon}{K_0}\right)
 =\begin{cases}
   1&:\varepsilon\equiv\text{\rm square}\npmod{\frak{p}_{K_0}},\\
  -1&:\text{\rm otherwise}
  \end{cases}
 \qquad
 (\varepsilon\in O_{K_0}^{\times})
$$
be the Legendre symbol of $K_0$. Then we have

\begin{prop}\label{prop:lambda-factor-of-tamely-ramified-galois-ext}
$$
 \lambda(K/F,\psi_F)
=\begin{cases}
  (-1)^{\frac{q^f-1}e\cdot\frac{e(e+2)}8}\cdot
   G_{\psi_{K_0}}(\left(\frac{\ast}{K_0}\right),
                   \varpi_0^{-(d(K_0)+1)})
    &:e=\text{\rm even},\\
  (-1)^{(f-1)d(F)}
    &:e=\text{\rm odd}
  \end{cases}
$$
where $\varpi_0$ is a prime element of $K_0$ such that 
$\varpi_0\in N_{K/K_0}(K^{\times})$.
\end{prop}

\subsection{$\gamma$-factors of admissible representations of Weil group}
\label{subsec:admissible-representation-of-weil-group}

\begin{dfn}\label{def:admisible-representation-of-weil-group}
The pair $(\Phi,V)$ is called an admissible representation of $W_F$ if
\begin{enumerate}
\item $V$ is a finite dimensional complex vector space and $\Phi$ is a
  group homomorphism of $W_F$ to $GL_{\Bbb C}(V)$,
\item $\text{\rm Ker}(\Phi)$ is an open subgroup of $W_F$,
\item $\Phi(\widetilde{\text{\rm Fr}})\in GL_{\Bbb C}(V)$ is
  semisimple.
\end{enumerate}
\end{dfn}

Let $(\Phi,V)$ be an admissible representation of $W_F$. Since 
$I_F=\text{\rm Gal}(F^{\text{\rm alg}}/F^{\text{\rm ur}})$ is a normal
subgroup of $W_F$, $\Phi(\widetilde{\text{\rm Fr}})\in GL_{\Bbb C}(V)$
keeps 
$$
 V^{I_F}=\{v\in V\mid\Phi(\sigma)v=v\;\forall\sigma\in I_F\}
$$
stable. Then the $L$-factor of $(\Phi,V)$ is defined by 
$$
 L(\Phi,s)=L(V,s)
 =\det\left(1-q^{-s}\cdot\Phi(\widetilde{\text{\rm Fr}})|_{V^{I_F}}
              \right)^{-1}.
$$
Since $\Phi:W_F\to GL_{\Bbb C}(V)$ is continuous group homomorphism,
we have the $\varepsilon$-factor $\varepsilon(\Phi,\psi,d(x),s)$ of
$\Phi$. Then the $\gamma$-factor of $(\Phi,V)$ is defined by
$$
 \gamma(\Phi,\psi,d(x),s)=\gamma(V,\psi,d(x),s)
 =\varepsilon(\Phi,\psi,d(x),s)\cdot
  \frac{L(\Phi\sphat,1-s)}
       {L(\Phi,s)}
$$
where $\Phi\sphat$ is the dual representation of $\Phi$.

\subsection{Symmetric tensor representation of $SL_2(\Bbb C)$}
\label{subsec:symmetric-tensor-representation-fo-sl(2)}
The complex special linear group $SL_2(\Bbb C)$ acts on the polynomial
ring $\Bbb C[X,Y]$ of two variables $X,Y$ by
$$
 g\cdot \varphi(X,Y)=\varphi((X,Y)g)
 \qquad
 (g\in SL_2(\Bbb C), \varphi(X,Y)\in\Bbb C[X,Y]).
$$
Let
$$
 \mathcal{P}_n
 =\langle X^n,X^{n-1}Y,\cdots,XY^{n-1},Y^n\rangle_{\Bbb C}
$$
be the subspace of $\Bbb C[X,Y]$ consisting of the homogeneous
polynomials of degree $n$. The action of $SL_2(\Bbb C)$ on
$\mathcal{P}_n$ defines the symmetric tensor representation 
$\text{\rm Sym}_n$ of degree $n+1$. The complex vector space
$\mathcal{P}_n$ has a non-degenerate bilinear form defined by
$$
 \langle\varphi,\psi\rangle
 =\left.
  \varphi\left(-\frac{\partial}
                     {\partial Y},\frac{\partial}
                                       {\partial X}\right)
  \psi(X,Y)\right|_{(X,Y)=(0,0)}\in\Bbb C
$$
for $\varphi,\psi\in\mathcal{P}_n$. This bilinear form is 
$SL_2(\Bbb C)$-invariant
$$
 \langle\text{\rm Sym}_n(g)\varphi,\text{\rm Sym}_n(g)\psi\rangle
 =\langle\varphi,\psi\rangle
 \quad
 (g\in SL_2(\Bbb C), \varphi,\psi\in\mathcal{P}_n)
$$
and
$$
 \langle\psi,\varphi\rangle=(-1)^n\langle\varphi,\psi\rangle
 \quad
 (\varphi,\psi\in\mathcal{P}_n).
$$
So we have group homomorphisms
$$
 \text{\rm Sym}_n:SL_2(\Bbb C)\to SO(\mathcal{P}_n)
 \;\;\text{\rm if $n$ is even}
$$
and
$$
 \text{\rm Sym}_n:SL_2(\Bbb C)\to Sp(\mathcal{P}_n)
 \;\;\text{\rm if $n$ is odd}.
$$

\subsection{Admissible representations of Weil-Deligne group}
\label{subsec:weil-deligne-group}
Fix a complex Lie group $\mathcal{G}$ such that the connected
component $\mathcal{G}^o$ is a reductive complex algebraic linear
group. Then the $\mathcal{G}^o$-conjugacy class of the group
homomorphisms 
$$
 \varphi:W_F\times SL_2(\Bbb C)\to\mathcal{G}
$$
such that
\begin{enumerate}
\item $I_F\cap\text{\rm Ker}(\varphi)$ is an open subgroup of $I_F$,
\item $\varphi(\widetilde{\text{\rm Fr}})\in\mathcal{G}$ is semi-simple,
\item $\varphi|_{SL_2(\Bbb C)}:SL_2(\Bbb C)\to\mathcal{G}^o$ is a
  morphism of complex linear algebraic group
\end{enumerate}
corresponds bijectively the equivalence classes of the triples 
$(\rho,\mathcal{G},N)$ where $N\in\text{\rm Lie}(\mathcal{G})$ is a
nilpotent element and 
$$
 \rho:W_F\to\mathcal{G}
$$
is a group homomorphism such that
\begin{enumerate}
\item $\rho|_{I_F}:I_F\to\mathcal{G}$ is continuous,
\item $\rho(\widetilde{\text{\rm Fr}})\in\mathcal{G}$ is semi-simple,
\item $\rho(g)N=|g|_F\cdot N$ for $\forall g\in W_F$ where
$$
 |\cdot|_F:W_F\xrightarrow{\text{\rm can.}}
           W_F/\overline{[W_F,W_F]}\xrightarrow{\text{\rm l.c.f.t.}}
           F^{\times}\xrightarrow{q^{-\text{\rm ord}_F(\cdot)}}
           \Bbb Q^{\times}
$$
\end{enumerate}
by the relations
$$
 \rho|_{I_F}=\varphi|_{I_F},
 \quad
 \rho(\widetilde{\text{\rm Fr}})
 =\varphi(\widetilde{\text{\rm Fr}})\cdot
  \varphi\begin{pmatrix}
          q^{-1/2}&0\\
          0&q^{1/2}
         \end{pmatrix},
 \quad
 N=d\varphi\begin{pmatrix}
            0&1\\
            0&0
           \end{pmatrix}
$$
(see \cite[Prop.2.2]{Gross-Reeder2010}). Here two triples 
$(\rho,\mathcal{G},N)$ and $(\rho^{\prime},\mathcal{G},N^{\prime})$ is
 equivalent if there exists a $g\in\mathcal{G}$ such that 
$\rho^{\prime}=g\rho g^{-1}$ and $N^{\prime}=\text{\rm Ad}(g)N$. 

The couple
 $(\varphi,\mathcal{G})$ or the triple $(\rho,\mathcal{G},N)$ is
 called an admissible representation of the Weil-Deligne group. 

Let $(r.V)$ be a continuous finite dimensional complex representation
of $\mathcal{G}$ which is algebraic on $\mathcal{G}^o$. Then the
$L$-factor associated with $(\varphi,\mathcal{G})$ and $(r,V)$ 
is defined by
$$
 L(\varphi,r,s)
 =\det\left(
   1-q^{-s}r\circ\rho(\widetilde{\text{\rm Fr}})|_{V_N^{I_F}}
             \right)^{-1},
$$
where $V_N=\{v\in V\mid dr(N)v=0\}$ and 
$$
 V_N^{I_F}
 =\{v\in V_N\mid r\circ\rho(\sigma)v=v\;\forall\sigma\in I_F\}.
$$
The $\varepsilon$-actor is defined by
$$
 \varepsilon(\varphi,r,\psi,d(x),s)
 =\varepsilon(r\circ\rho,\psi,d(x),s)\cdot
  \det\left(-q^{-s}r\circ\rho(\widetilde{\text{\r Fr}})
             |_{V^{I_F}/V_N^{I_F}}\right)
$$
where $\varepsilon(r\circ\rho,\psi,d(x),s)$ is the $\varepsilon$-factor
of the representation $(r\circ\rho,V)$ of $W_F$ defined in the
subsection \ref{subsec:admissible-representation-of-weil-group}. 
Finally the $\gamma$-factor is defined by 
$$
 \gamma(\varphi,r,\psi,d(x),s)
 =\varepsilon(\varphi,r,\psi,d(x),s)\cdot
   \frac{L(\varphi,r^{\vee},1-s)}
        {L(\varphi,r,s)}
$$
where $r^{\vee}$ is the dual representation of $r$. 

Let $\text{\rm Sym}_n$ be the symmetric tensor representation of
$SL_2(\Bbb C)$ of degree $n+1$. Then the $W_F\times SL_2(\Bbb
C)$-module $V$ has a decomposition
$$
 V=\bigoplus_{n=0}^{\infty}V_n\otimes\text{\rm Sym}_n
$$
where $V_n$ is a $W_F$-module. Then we have
$$
 V_N^{I_F}
 =\bigoplus_{n=0}^{\infty}V_n^{I_F}\otimes\text{\rm Sym}_{n,N}
$$
where $\text{\rm Sym}_{n,N}$ is the highest part of 
$\text{\rm Sym}_n$. 
Since $r\circ\rho(\widetilde{\text{\rm Fr}})$ act on 
$V_n\otimes\text{\rm Sym}_{n,N}$ by 
$q^{-n/2}r\circ\varphi(\widetilde{\text{\rm Fr}})$, we have
$$
 L(\varphi,r,s)
 =\prod_{n=0}^{\infty}\det\left(
   1-q^{-(s+n/2)}r\circ\varphi(\widetilde{\text{\rm Fr}})|_{V_n^{I_F}}
                                 \right)^{-1}.
$$
If the Haar measure $d(x)$ on the additive group $F$ and the additive
character $\psi:F\to\Bbb C^{\times}$ are normalized so
that $\int_{O_F}d(x)=1$ and
$$
 \{x\in F\mid\psi(xO_F)=1\}=O_F,
$$
then we have
$$
 \varepsilon(\varphi,r,\psi,d(x),s)
 =w(\varphi,r)\cdot q^{a(\varphi,r)(1/2-s)}
$$
where
$$
 w(\varphi,r)=\prod_{n=0}^{\infty}w(V_n)^{n+1}\cdot
              \prod_{n=1}^{\infty}\det\left(
               -\varphi(\widetilde{\text{\rm Fr}})|_{V_n^{I_F}}
                                              \right)^n
$$
and
$$
 a(\varphi,r)=\sum_{n=0}^{\infty}(n+1)a(V_n)
             +\sum_{n=1}^{\infty}n\cdot\dim V_n^{I_F}.
$$
If $\varphi|_{SL_2(\Bbb C)}=1$, then $V_n=0$ for all $n>0$ and we have
$$
 w(\varphi,r)=w(r\circ\varphi)=w(r\circ\rho),
 \quad
 a(\varphi,r)=a(r\circ\varphi)=a(r\circ\rho).
$$

Sendai 980-0845, Japan\\
Miyagi University of Education\\
Department of Mathematics
\end{document}